\documentclass[a4paper,11pt]{article}
\usepackage{bbm}
\usepackage{}
\usepackage{mathrsfs}
\usepackage{amsfonts}
\usepackage{amsfonts}
\usepackage{amssymb}
\usepackage{mathrsfs}
\usepackage{indentfirst,latexsym,bm,amsmath,amssymb,amsthm}
\usepackage{xcolor}
\usepackage[
                   bookmarksnumbered=true,
                   bookmarksopen=true, 
                   pdfauthor=kmc,
                   pdfcreator={LaTex with hyperref package + WinEdt},
                   pdftitle=trudge,
                   colorlinks,%
                   citecolor=blue,
                   linkcolor=blue,%
                   urlcolor=blue,
                   hyperindex,%
                   plainpages=false,%
                   pdfstartview=FitH,
                   linktocpage=true,
                   dvipdfm%
                   ]{hyperref}
\usepackage{indentfirst,latexsym,bm,amsmath,amssymb,amsthm}
\usepackage[dvips]{graphicx}

\makeatletter\@addtoreset{equation}{section} \makeatother
\newtheorem{thm}{Theorem}[section]
\newtheorem{Lemma}{Lemma}[section]
\newtheorem{proposition}{Proposition}[section]

\newtheorem{rem}{Remark}[section]

 

\title{Scattering and Rigidity for Nonlinear Elastic Waves}

\author{Dongbing Zha\thanks{Department of Mathematics, Donghua University, Shanghai 201620, PR China.{ E-mail address: ZhaDongbing@163.com} }}

\begin{document}

\maketitle
\begin{abstract}

For the Cauchy problem of nonlinear elastic wave equations of three dimensional isotropic, homogeneous and hyperelastic materials satisfying the null condition, global existence of classical solutions with small initial data was proved in R. Agemi (\href{https://link.springer.com/article/10.1007/s002220000084}{Invent. Math. 142 (2000) 225--250}) and T. C. Sideris (\href{https://www.jstor.org/stable/121050?origin=crossref}{Ann. Math. 151
 (2000) 849--874}), independently. In this paper, we will consider the asymptotic behavior of global solutions. We first show that the global solution will scatter, i.e., it will converge to some solution of linear elastic wave equations as time tends to infinity, in the energy sense.
We also prove the following rigidity result: if the scattering data vanish, then the global solution will also vanish identically.
 The variational structure of the system will play a  key role in our argument.
 \\
\emph{keywords}: Nonlinear elastic waves; null condition; scattering; rigidity
\end{abstract}
\pagestyle{plain} \pagenumbering{arabic}

\section{ Introduction  }
For isotropic, homogeneous and hyperelastic materials, the motion for the displacement is governed by the nonlinear elastic wave equation which is a second-order quasilinear hyperbolic system.
Some physical backgrounds of nonlinear elastic waves can be found in
 \cite{MR936420}, \cite{Gurtin81} and \cite{MR2985747}.\par

Researches on long time existence of classical solutions for nonlinear elastic waves can trace back to Fritz John's seminal works on elastodynamics \cite{MR446001, MR676668, John84, John88}, which is also the main motivation for his pioneering works on nonlinear wave equations  (see \cite{MR1066694, Klainerman98}).
  For the Cauchy problem of three dimensional nonlinear elastic waves,
John \cite{John84} proved that in the radially
symmetric case, a genuine nonlinearity condition will lead to the formation of
singularities for small initial data (see also \cite{MR676668}). John \cite{John88} also showed that the equations have almost global classical solutions
for small initial data. Then Klainerman and~Sideris \cite{Klainerman96} simplified John's proof. See also other simplified proofs in \cite{almost}, \cite{MR3014806} and \cite{MR3650329}. Agemi \cite{Agemi00} and Sideris \cite{Sideris00}
proved independently that for certain classes of materials satisfying a null condition, which is the complement of John's genuine nonlinearity condition,  there exist global classical solutions with small initial data.
Both null conditions in \cite{Agemi00} and \cite{Sideris00} are based on the John-Shatah observation in \cite{MR1066694} for the null condition of nonlinear wave equations, which is introduced by Klainerman \cite{Klainerman82} (see also \cite{Christodoulou86, Klainerman86}).
See also a previous result in \cite{Sideris96} and an alternative proof in \cite{MR4188824}, for the global existence of classical solutions with small initial data. For large initial data, Tahvildar-Zadeh
\cite{MR1648985} proved that singularities will always form no matter whether the null condition
holds or not. 
Some low regularity global existence results under the null condition in the radially symmetric case are given in \cite{MR4149687}. 
 The global existence of classical solutions to two dimensional elastic waves satisfying null condition with small initial data is still an open problem. Some partial results in two dimensional case can be found in \cite{MR4387288}, \cite{MR3485863}
and \cite{MR4029005}.

For three dimensional nonlinear elastic wave equations with null condition,
though the global existence of classical solutions has been known since \cite{Agemi00} and \cite{Sideris00}, how to describe the asymptotic behavior of global solutions is still an open problem.
The aim of this manuscript is to resolve this problem. We show that the global solution will scatter, that is, it will converge to some solution of linear elastic wave equations as time tends to infinity, in the energy sense. We also prove the following rigidity result, that is, if the scattering data vanish, then the global solution will also vanish identically.
We will see that the variational structure of the system will play a key role in our argument.

The outline of this manuscript is as follows. The remainder of this introduction will be devoted to the description of
the basic notation which will be used in the sequel, description of the equations of motion and a statement of the main result, i.e., Theorem \ref{mainthm}. Section \ref{scci34} introduces some necessary tools, including some commutation relations, decay estimates for null form nonlinearity and weighted Sobolev inequalities. Then the scattering part and the rigidity part of Theorem \ref{mainthm} will be proved in Section \ref{SEC3333}
 and Section \ref{sbj8978}, respectively.

\subsection{Notation}
Denote the space gradient and space-time gradient by $\nabla=(\partial_1,\partial_2,\partial_3)$ and
$
\partial =(\partial_t,\nabla),
$
respectively.
The angular momentum operators (generator of the spatial rotation) are the vector fields
$
\Omega=(\Omega_{ij}: 1\leq i<j\leq 3),
$
where
\begin{align}
\Omega_{ij}=x_i\partial_j-x_j\partial_i.
\end{align}
Now we define $\widetilde{\Omega}_{ij} ( 1\leq i<j\leq 3) $ as follows. For vector-valued function $u: \mathbb{R}^3\longrightarrow \mathbb{R}^3$, they are the generators of
simultaneous rotations:
\begin{equation}\label{below222}
\widetilde{\Omega}_{ij}u=
\Omega_{ij}u+U_{ij}u,
\end{equation}
where $U_{ij}=e_i\otimes e_j-e_j\otimes e_i$, $\{e_i\}_{i=1}^{3}$ is the standard basis
on $\mathbb{R}^3$, and for scalar-valued function $\phi: \mathbb{R}^3\longrightarrow \mathbb{R}$, just set
\begin{equation}\label{beloddw222}
\widetilde{\Omega}_{ij}\phi=
\Omega_{ij}\phi.
\end{equation}
Denote  $\widetilde{\Omega}=(\widetilde{\Omega}_{ij}: 1\leq i<j\leq 3)$.
The scaling operator is
\begin{align}\label{i22nview1}
S=t\partial_t+r\partial_r,
\end{align}
where $r=|x|,~ \partial_r=\omega\cdot \nabla, ~\omega=(\omega_1,\omega_2,\omega_3), ~\omega_i=x_i/r, i=1,2,3$,
and we will use
\begin{align}\label{i22ssnview1}
\widetilde{S}=S-1.
\end{align}
Denote the collection of vector fields by $\Gamma=(\partial_t, \nabla,\widetilde{\Omega}, \widetilde{S})=(\Gamma_0,\dots,\Gamma_7)$.
 By $\Gamma^{a}$, $a=(a_1,\dots,a_k)$, we denote an ordered product of $k=|a|$ vector fields $\Gamma_{a_1}\cdots \Gamma_{a_k}$.
\par
Denote the basic energy corresponding to the linear elastic wave operator (see Section \ref{sec12}) by
\begin{align}\label{xhsssk90}
\mathcal {E}_1(u(t))=\frac{1}{2}\int_{\mathbb{R}^3}\big(|\partial_tu(t,x)|^2+c_2^2|\nabla u(t,x)|^2+(c_1^2-c_2^2)(\nabla \cdot u(t,x))^2\big){\rm d}x,
\end{align}
where the constants $0<c_2<c_1$,
and the corresponding higher order version by
\begin{align}
\mathcal {E}_{\kappa}(u(t))=\sum_{|\alpha|\leq \kappa-1}\mathcal {E}_1(\nabla^{\alpha}u(t)).
\end{align}
Denote also a first order energy by
\begin{align}
E_1(u(t))=\int_{\mathbb{R}^3}\big(|\partial_tu(t,x)|^2+|\nabla u(t,x)|^2\big){\rm d}x,
\end{align}
which is equivalent to $\mathcal {E}_1(u(t))$,
and the higher order version by
\begin{align}
E_{\kappa}(u(t))=\sum_{|a|\leq \kappa-1}E_1(\Gamma^{a}u(t)).
\end{align}

\par
The solution will be constructed in the space $\dot{H}^{\kappa}_{\Gamma}(T)$, which is the closure of the set $C^{\infty}\big([0,T);C_{c}^{\infty}(\mathbb{R}^3;\mathbb{R}^3)\big)$  in the norm $\sup\limits_{0\leq t<T}E_{\kappa}^{1/2}(u(t))$.\par
Set
\begin{align}
\Lambda=(\nabla,\widetilde{\Omega},r\partial_r-1)=(\Lambda_1,\dots,\Lambda_7).
\end{align}
 Define the time-independent spaces
\begin{align}
H^{\kappa}_{\Lambda}=\{u\in L^2(\mathbb{R}^3;\mathbb{R}^3): \Lambda^{a}u\in L^2(\mathbb{R}^3;\mathbb{R}^3),|a|\leq \kappa\}
\end{align}
with the norm
\begin{align}
\|u\|_{H^{\kappa}_{\Lambda}}=\sum_{|a|\leq \kappa}\|\Lambda^{a}u\|_{L^2}.
\end{align}
We also use the notation
\begin{align}
\mathcal {H}^{\kappa}=\{(u,v): \nabla u\in {H}^{\kappa-1}, v\in  {H}^{\kappa-1} \}
\end{align}
and the corresponding the norm
\begin{align}
\|(u,v)\|_{\mathcal {H}^{\kappa}}=\|\nabla u\|_{{H}^{\kappa-1}}+\| v\|_{{H}^{\kappa-1}}.
\end{align}
%
Denote the Riesz transformation by
\begin{align}
R_k=\frac{\partial_k}{\sqrt{-\Delta}},~k=1,2,3.
\end{align}
We also employ the notation $R=(R_1,R_2,R_3)$.
We will use the Helmholtz decomposition, which projects any vector field onto curl-free and
divergence-free components.
Specifically speaking,
for any function~$u\in H^{2}(\mathbb{R}^{3};\mathbb{R}^{3})$, we have
\begin{align}\label{hodge1}
u=u_{cf}+u_{df},
\end{align}
with
\begin{align}\label{hodge2}
u_{cf}=-R(R\cdot u),~~     u_{df}=R\wedge(R\wedge u),
\end{align}
$\wedge$ being the usual vector cross product.
It holds that
\begin{align}\label{hodge3}
\nabla&\wedge u_{cf}=0,~~  \nabla\cdot u_{df}=0,
\end{align}
and
\begin{align}\label{hodge4567}
\|\nabla^{\alpha} u\|^2_{L^2}&=\|\nabla^{\alpha} u_{cf}\|^2_{L^2}+\|\nabla^{\alpha} u_{df}\|^2_{L^2},
\end{align}
for $|\alpha|\leq 2$.

Now we will define some weighted $L^2$ norm, for scalar functions and vector functions, respectively.
For scalar function $\phi$, we define
\begin{align}
 \mathcal {X}_{\kappa}(\phi(t))=\sum_{\beta=0}^{3}\sum_{l=1}^{3}\sum_{|a|\leq \kappa-2}\big(\|\langle c_1t-r\rangle \partial_{\beta}\partial_{l}\Gamma^{a}\phi(t)\|_{L^2},
 \end{align}
 where $\langle \cdot\rangle=(1+|\cdot|^2)^{1/2}$.
For vector function $u$, using the Helmholtz decomposition, we will use the following weighted $L^2$ norm
 \begin{align}
 \mathcal {X}_{\kappa}(u(t))=\sum_{\beta=0}^{3}\sum_{l=1}^{3}\sum_{|a|\leq \kappa-2}\big(\|\langle c_1t-r\rangle \partial_{\beta}\partial_{l}\Gamma^{a}u_{cf}(t)\|_{L^2}+\|\langle c_2t-r\rangle \partial_{\beta}\partial_{l}\Gamma^{a}u_{df}(t)\|_{L^2}\big).
 \end{align}
 Note that the above definition of weighted $L^2$ norm $\mathcal {X}_{\kappa}(u(t))$ is different from the corresponding one in \cite{Sideris00} (see (1.6) in \cite{Sideris00}), where a local decomposition is employed.
\subsection{The equations of motion}\label{sec12}
We consider the equations of motion for 3-D homogeneous, isotropic and hyperelastic elastic waves. First we have the Lagrangian
 \begin{align}\label{lagrange}
 \mathscr{L}(u)=\iint \big(\frac{1}{2}|u_t|^2-W(\nabla u)\big)~ {\rm d}x{\rm d}t,
\end{align}
where $u(t,x)=(u^{1}(t,x),u^{2}(t,x),u^{3}(t,x))$ denotes the displacement vector from the
reference configuration and $W$ is the stored energy function. Since small solutions will be considered, we can write
\begin{align}
W(\nabla u)=l_2(\nabla u)+l_3(\nabla u)+ h.o.t.,
\end{align}
with $l_2(\nabla u)$ and $l_3(\nabla u)$ stand for the quadratic and cubic term in $\nabla u$, respectively, and $h.o.t.$ denotes higher order terms.
Using the frame indifference and isotropic assumption, we can get (see \cite{Agemi00, Sideris96, MR4188824})
\begin{align}\label{L2}
 l_2(\nabla u)=\frac{c_2^2}{2}|\nabla  u|^2+\frac{c_1^2-c_2^2}{2}(\nabla \cdot u)^2,
 \end{align}
 where the material constants $c_1$ (pressure wave speed) and $c_2$ (shear wave speed) satisfy $0<c_2<c_1$,
 and\footnote{Repeated indices are always summed.}
\begin{align}\label{sanjie}
 l_3(\nabla u)&=d_1(\nabla \cdot u)^3+d_2(\nabla \cdot u)|\nabla\wedge u|^2+d_3(\nabla \cdot u)Q_{ij}(u^i,u^j)\nonumber\\
 &~+d_4(\partial_{k}u^{j})Q_{ij}(u^{i},u^{k})+d_5(\partial_{k}u^{j})Q_{ik}(u^{i},u^{j}),
 \end{align}
where $d_i~(i=1,\dots,5)$ are some constants which depend only on the stored energy function, and the null form
\begin{equation}
Q_{ij}(f,g)=\partial_i f\partial_jg-\partial_jf\partial_ig.
\end{equation}.

By Hamilton's principle we get the nonlinear elastic wave equation in 3-D as follows:\footnote{We truncate the nonlinearity at the quadratic level, because the higher order
terms have no essential influence on the discussion of the global existence and asymptotic behavior of solutions with small amplitude.}
 \begin{align}\label{Cauchy}
Lu=N( u, u).
 \end{align}
Here the linear elastic wave operator
 \begin{align}
L=  (\partial_t^2-c_2^2\Delta)I -(c_1^2-c_2^2)\nabla\otimes\nabla,
 \end{align}
the nonlinearity
 \begin{align}
N(u,v)= N_0(u,v)+N_1(u,v)+N_2(u,v)+N_3(u,v),
 \end{align}
 with
 \begin{align}\label{nqsddd}
N_0(u,v)&=3{\color{blue}{d_1}}\nabla \big((\nabla\cdot u)~(\nabla\cdot v)\big),\\\label{nqsddd22}
N_1(u,v)&=d_2\nabla\big((\nabla\wedge u)\cdot (\nabla\wedge v)\big),\\\label{nqsddd33}
N_2(u,v)=-d_2\nabla\wedge\big((\nabla&\cdot u) (\nabla\wedge v)\big)-d_2\nabla\wedge\big((\nabla\cdot v) (\nabla\wedge u)\big)
 \end{align}
 and
 \begin{align}\label{xuyaoghujjk}
N_3(u, v)^{i}&=(d_3+\frac{d_4}{2})\big[Q_{ij}(\partial_ku^{k},v^{j})+Q_{ij}(\partial_kv^{k},u^{j})-Q_{jk}(\partial_iu^{k},v^{j})-Q_{jk}(\partial_iv^{k},u^{j})\big]\nonumber\\
&+\frac{d_5}{2}\big[Q_{ij}(\partial_ju^{k},v^{k})+Q_{ij}(\partial_jv^{k},u^{k})+2Q_{jk}(\partial_ju^{i},v^{k})+2Q_{jk}(\partial_jv^{i},u^{k}) \big]\nonumber\\
&-\frac{d_5}{2}\big[ Q_{jk}(\partial_ju^{k},v^{i})+Q_{jk}(\partial_jv^{k},u^{i})\big].
 \end{align}

  We say that the nonlinear elastic wave equation \eqref{Cauchy} satisfies the null condition if
 \begin{align}\label{null1111}
 d_1=0.
 \end{align}
 See \cite{Agemi00} and \cite{Sideris00}. In view of \eqref{sanjie}, we can see that the null condition \eqref{null1111} is just used to rule out the term $(\nabla\cdot u)^3$ in the stored energy function.  Thus we can identify whether materials satisfy the null condition by checking the cubic term in the stored energy function directly.
\subsection{Main result}
Consider the Cauchy problem for \eqref{Cauchy} with initial data
\begin{equation}\label{Cajujiniu}
(u, u_t)|_{t=0}=(u_0, u_1).
\end{equation}
As a refinement for the global existence results in  \cite{Agemi00} and \cite{Sideris00}, we have the following (see \cite{MR4188824})

\begin{proposition}\label{zhuyaodingli}
Assume that the null condition \eqref{null1111} is satisfied and
\begin{align}\label{INITIALD}
 u_0\in H^{\kappa}_{\Lambda},~~u_1\in H^{\kappa-1}_{\Lambda},~\kappa\geq 7.
\end{align}
Then the Cauchy problem \eqref{Cauchy}--\eqref{Cajujiniu}
admits a unique global solution $u\in \dot{H}^{\kappa}_{\Gamma}(T)$ for every $T>0$, if
\begin{align}\label{rt5633}
E_{\kappa-1}(u(0))\exp{\big[C_0E^{1/2}_{\kappa}(u(0))\big]}\leq \varepsilon^2
\end{align}
and $\varepsilon$ is sufficiently small, depending on $C_0$. The global solution satisfies the bounds
\begin{align}\label{Lowenergy}
\mathcal {X}_{l}(u(t))\leq CE^{1/2}_{l}(u(t))\leq C\varepsilon,~l\leq \kappa-1,
\end{align}
and
\begin{align}
\mathcal {X}_{\kappa}(u(t))\leq C E^{1/2}_{\kappa}(u(t))\leq 2CE^{1/2}_{\kappa}(u(0))\langle t\rangle^{\sqrt{C_0\varepsilon}},
\end{align}
for every $t\geq 0$.
\end{proposition}
\begin{rem}
The first global existence result for nonlinear elastic waves is due to {\rm \cite{Sideris96}}, under the stronger null condition $d_1=d_2=0$ (see \eqref{sanjie}), which forces the cancellation of
all nonlinear wave interactions to first order along the characteristic cones. Then {\rm \cite{Agemi00}} and {\rm \cite{Sideris00}} show the global existence, under the null condition $d_1=0$, which only requires the cancellation of nonlinear wave interactions among individual
wave families. The approach in {\rm \cite{Agemi00}} concerns some pointwise estimates for the fundamental solution of linea elastic waves, along John's original spirit in {\rm \cite{John88}}.
The method in {\rm \cite{Sideris00}} is based on some weighted $L^2$ estimates (Klainerman-Sideris type estimates), which started from {\rm \cite{Klainerman96}}, and some local decomposition.
 By using the original Klainerman-Sideris estimate for wave operators in {\rm \cite{Klainerman96}} and the nonlocal Helmholtz decomposition, in {\rm \cite{MR4188824}} we provide an alternative proof for the global existence results in {\rm \cite{Agemi00}} and {\rm \cite{Sideris00}}. Due to the divergence-curl structure in the nonlinearity (see \eqref{nqsddd22} and \eqref{nqsddd33}), it seems that the Helmholtz decomposition is a suitable way to treat the asymptotic behavior of global solutions.
\end{rem}

Now, for a function $u\in \dot{H}^{2}_{\Gamma}(T)$ for every $T>0$, we say that it is asymptotically free (in the energy sense),
if for the solution $\overline{u}$ to the homogeneous linear elastic wave equation
\begin{align}
L\overline{u}=0~~~{\text{on}}~~\mathbb{R}^{+}\times \mathbb{R}^3
\end{align}
for suitable initial data
\begin{align}\label{dddddsss}
(\overline{u}, \overline{u}_t)|_{t=0}=(\overline{u}_0, \overline{u}_1)\in \mathcal {H}^{2}
\end{align}
there holds
\begin{align}
\|\partial u(t)-\partial \overline{u}(t)\|_{H^1}\longrightarrow 0,~~~\text{as}~~t\longrightarrow \infty.
\end{align}

By definition, a global solution $u$ to nonlinear elastic wave equation \eqref{Cauchy} scatters if it is asymptotically free.  The corresponding initial data \eqref{dddddsss} is called the \lq\lq scattering data".
 %

The main result of this manuscript is the following
\begin{thm}\label{mainthm}
For nonlinear elastic wave equation \eqref{Cauchy} satisfying null condition \eqref{null1111},
the global solution
constructed in Proposition {\rm{\ref{zhuyaodingli}}}
scatters, and if the scattering data vanish, then the global solution will also vanish identically.
\end{thm}

\begin{rem}
We point out that some rigidity results are proved in Li and Yu {\rm \cite{MR4223342}}, in the contexts of Alfv\'{e}n waves in MHD, which is governed by a transport type system.
This is one main motivation for our consideration of the rigidity result in Theorem {\rm \ref{mainthm}}, but the approach of our proof is different from the corresponding one in it. Some other related works can be found in {\rm \cite{MR3415691}, \cite{MR2461426}, \cite{MR4258125}, \cite{Mengniliss}}, etc.
\end{rem}
\begin{rem}
We conjecture that the nonlinear elastic wave system admits the following inverse scattering property: the scattering data can determine the global solution uniquely. It is obvious that the inverse scattering property can imply the rigidity result in Theorem {\rm \ref{mainthm}}.
\end{rem}

\section{Preliminaries}\label{scci34}
In this section, we collect some necessary tools, including some commutation relations, decay estimates for null form nonlinearity and weighted Sobolev inequalities. They will be used frequently in the proof of Theorem \ref{mainthm}.
\begin{Lemma}\label{comuuio}
We have the following commutation relations
\begin{align}
[\Omega_{ij},\partial_k]&=\delta_{jk}\partial_i-\delta_{ik}\partial_j,~[S,\partial_k]=-\partial_k,~~[S,\partial_t]=-\partial_t,\\
~~[\Omega_{ij},\partial_r]&=-r^{-1}\Omega_{ij},~~~~~~[S,\partial_r]=-\partial_r,~~[\partial_k,\partial_r]=r^{-1}(\partial_k-\omega_k\partial_r).
\end{align}
\end{Lemma}

   \begin{Lemma}\label{decay}
We have
\begin{align}\label{hjkkmmiddddd890}
\big|Q_{ij}(u^{l},v^{m})\big|&\leq  \frac{C}{r}\sum_{|a|\leq 1}\big(|\widetilde{\Omega}^{a}u||\nabla v|+|\widetilde{\Omega}^{a}v||\nabla u|\big),\\\label{hjkkiddd890}
\big|Q_{ij}(\partial_ku^{l},v^{m})\big|&\leq  \frac{C}{r}\sum_{|a|\leq 1}\big(|\widetilde{\Omega}^{a}u||\nabla^2v|+|\widetilde{\Omega}^{a}v||\nabla^2u|+|\nabla\widetilde{\Omega}^{a}u||\nabla v|+|\nabla\widetilde{\Omega}^{a}v||\nabla u|\big),
\end{align}
for $1\leq i,j,k,l,m\leq 3$,
and
\begin{align}\label{DECAY}
|N_3(u,v)|\leq \frac{C}{r}\sum_{|a|\leq 1}\big(|\widetilde{\Omega}^{a}u||\nabla^2v|+|\widetilde{\Omega}^{a}v||\nabla^2u|+|\nabla\widetilde{\Omega}^{a}u||\nabla v|+|\nabla\widetilde{\Omega}^{a}v||\nabla u|\big).
\end{align}
\end{Lemma}
\begin{proof}
It follows from the radial-angular decomposition
\begin{align}
\nabla=\omega\partial_r-\frac{\omega\wedge \Omega}{r}
\end{align}
that
\begin{align}
&\partial_{i}f\partial_{j}g=\omega_{i}\partial_rf\partial_jg-\frac{(\omega\wedge \Omega)_{i}}{r}f\partial_jg\nonumber\\
&
=\omega_{i}\omega_{j}\partial_rf\partial_rg-\omega_{i}\partial_rf\frac{(\omega\wedge \Omega)_{j}}{r}g-\frac{(\omega\wedge \Omega)_{i}}{r}f\partial_jg.
\end{align}
Thus we have
\begin{align}
&Q_{ij}(f,g)=\partial_{i}f\partial_{j}g-\partial_{j}f\partial_{i}g\nonumber\\
&
=-\omega_{i}\partial_rf\frac{(\omega\wedge \Omega)_{j}}{r}g-\frac{(\omega\wedge \Omega)_{i}}{r}f\partial_jg+\omega_{j}\partial_rf\frac{(\omega\wedge \Omega)_{i}}{r}g+\frac{(\omega\wedge \Omega)_{j}}{r}f\partial_ig,
\end{align}
which implies
\begin{equation}\label{nuljk99}
\big|Q_{ij}(f,g)\big|\leq \frac{C}{r}\big(|\nabla f||{\Omega}g|+|\nabla g ||{\Omega}f|\big).
\end{equation}
Then by \eqref{nuljk99}, \eqref{below222} and Lemma \ref{comuuio}, we get \eqref{hjkkmmiddddd890} and \eqref{hjkkiddd890}.
In view of \eqref{xuyaoghujjk}, \eqref{hjkkiddd890} implies \eqref{DECAY}.
\end{proof}

    \begin{Lemma}
For vector function $u$, we have
\begin{align}
\label{gao1}
\langle r\rangle^{1/2}|\Gamma^{a}u(t,x)|&\leq CE^{1/2}_{\kappa}(u(t)),~|a|+2\leq \kappa,\\\label{gao2}
\langle r\rangle|\partial\Gamma^{a}u(t,x)|&\leq CE^{1/2}_{\kappa}(u(t)),~|a|+3\leq \kappa,\\\label{gao300}
\langle r\rangle^{1/2}\langle c_{1}t-r\rangle|\partial\Gamma^{a}u_{cf}(t,x)|&+\langle r\rangle^{1/2}\langle c_{2}t-r\rangle|\partial\Gamma^{a}u_{df}(t,x)|\nonumber\\
~~~~~\leq CE^{1/2}_{\kappa}(u(t))&+C\mathcal {X}_{\kappa}(u(t)),~~~~|a|+3\leq \kappa,\\\label{gao3}
\langle r\rangle\langle c_{1}t-r\rangle^{1/2}|\partial\Gamma^{a}u_{cf}(t,x)|&+\langle r\rangle\langle c_{2}t-r\rangle^{1/2}|\partial\Gamma^{a}u_{df}(t,x)|\nonumber\\
~~~~~\leq CE^{1/2}_{\kappa}(u(t))&+C\mathcal {X}_{\kappa}(u(t)),~~~~|a|+3\leq \kappa,\\\label{gao4}
\langle r\rangle\langle c_{1}t-r\rangle|\partial\nabla\Gamma^{a}u_{cf}(t,x)|&+\langle r\rangle\langle c_{2}t-r\rangle|\partial\nabla\Gamma^{a}u_{df}(t,x)|\nonumber\\
~~~~~&\leq C\mathcal {X}_{\kappa}(u(t)),~~~|a|+4\leq \kappa.
\end{align}
Similarly, for scalar function $\phi$, we also have
\begin{align}\label{Sgao3}
\langle r\rangle\langle c_{1}t-r\rangle^{1/2}|\partial\Gamma^{a}\phi(t,x)|&\leq CE^{1/2}_{\kappa}(\phi(t))+C\mathcal {X}_{\kappa}(\phi(t)),~|a|+3\leq \kappa,\\\label{Sgao4}
\langle r\rangle\langle c_{1}t-r\rangle|\partial\nabla\Gamma^{a}\phi(t,x)|&\leq C\mathcal {X}_{\kappa}(\phi(t)),~~~~~~~~~~~~~~~~~~~|a|+4\leq \kappa.
\end{align}
\end{Lemma}

\begin{proof}
For the proofs of \eqref{gao1} and \eqref{gao2}, we refer the reader to \cite{Sideris00}. As for \eqref{Sgao3} and \eqref{Sgao4},  see \cite{Sideris01}. While the proofs of
\eqref{gao300}, \eqref{gao3} and \eqref{gao4} can be found in \cite{MR4188824}.
\end{proof}
\section{Proof of Theorem \ref{mainthm}: scattering}\label{SEC3333}
Due to the highly coupled feature of elastic waves (even in the linear level),
it seems that treating the scattering problem for \eqref{Cauchy} directly  is very difficult. Our main strategy is to split the global solution for \eqref{Cauchy} into two parts,  then show that each part is asymptotically free.  This simple observation is the first key point in our argument. We will see that the variational structure of the system plays a key role in our treatment.

Specifically speaking,
assume that $u$ is the global solution to the Cauchy problem \eqref{Cauchy}--\eqref{Cajujiniu}, constructed in Theorem {\rm{\ref{zhuyaodingli}}} (from now on, we just fix $\kappa=7$ in Theorem {\rm{\ref{zhuyaodingli}}}).
Then we write
$
u=w+v,
$
where
$w$ satisfies
\begin{equation}\label{gloV2}
\begin{cases}
Lw=G(t,x),~~~{\text{on}}~~\mathbb{R}^{+}\times \mathbb{R}^3,\\
(w, w_t)|_{t=0}=(u_0, u_1),
\end{cases}
\end{equation}
$v$ satisfies
\begin{equation}\label{gloV}
\begin{cases}
Lv=F(t,x),~~~{\text{on}}~~\mathbb{R}^{+}\times \mathbb{R}^3,\\
(v, v_t)|_{t=0}=(0, 0),
\end{cases}
\end{equation}
and
\begin{align}
G(t,x)&=N_2(u,u)+N_3(u,u),\\\label{FF}
F(t,x)&=N_1(u,u).
\end{align}

We will prove
\begin{proposition}\label{mingti2}
The global solution $w$ to the Cauchy problem \eqref{gloV2} is asymptotically free.
\end{proposition}
\begin{proposition}\label{mingti1}
The global solution $v$ to the Cauchy problem \eqref{gloV} is asymptotically free.
\end{proposition}

The proofs of Proposition \ref{mingti2} and Proposition \ref{mingti1} will be provided in Section \ref{sdfffffff} and Section \ref{sdffffdddfff}, respectively. In the remainder of this section, we will give the proof of scattering part of Theorem \ref{mainthm} based on  Proposition \ref{mingti2} and Proposition \ref{mingti1}.

   For nonlinear elastic wave equation \eqref{Cauchy} satisfying null condition \eqref{null1111}, assume that $u$ is
the global solution constructed in Proposition {\rm{\ref{zhuyaodingli}}}. We have
\begin{equation}\label{WW1}
u=w+v,
\end{equation}
 where $w$ and $v$ satisfy \eqref{gloV2}  and \eqref{gloV} respectively. According to Proposition \ref{mingti2}, $w$ is asymptotically free. That is, we have
\begin{equation}\label{XXhjiooldd}
\|\partial w(t)-\partial \overline{w}(t)\|_{H^1}\longrightarrow 0,~~~\text{as}~~t\longrightarrow \infty,
\end{equation}
where $\overline{w}$ satisfies
\begin{equation}
L\overline{w}=0~~~{\text{on}}~~\mathbb{R}^{+}\times \mathbb{R}^3
\end{equation}
for some initial data
\begin{equation}
(\overline{w}, \overline{w}_t)|_{t=0}=(\overline{w}_0, \overline{w}_1)\in \mathcal {H}^{2}.
\end{equation}
Similarly, by Proposition \ref{mingti1}, $v$ is asymptotically free. That is, we have
\begin{equation}\label{hjiooldd}
\|\partial v(t)-\partial \overline{v}(t)\|_{H^1}\longrightarrow 0,~~~\text{as}~~t\longrightarrow \infty,
\end{equation}
where $\overline{v}$ satisfies
\begin{equation}
L\overline{v}=0~~~{\text{on}}~~\mathbb{R}^{+}\times \mathbb{R}^3
\end{equation}
for some initial data
\begin{equation}
(\overline{v}, \overline{v}_t)|_{t=0}=(\overline{v}_0, \overline{v}_1)\in \mathcal {H}^{2}.
\end{equation}
Now set
\begin{equation}
\overline{u}_0=\overline{v}_0+\overline{w}_0,~~~~\overline{u}_1=\overline{v}_1+\overline{w}_1,
\end{equation}
and
\begin{equation}\label{xddhj897}
\overline{u}=\overline{v}+\overline{w}.
\end{equation}
We have
\begin{equation}
\begin{cases}
L\overline{u}=0~~~{\text{on}}~~\mathbb{R}^{+}\times \mathbb{R}^3,\\
(\overline{u}, \overline{u}_t)|_{t=0}=(\overline{u}_0, \overline{u}_1)\in \mathcal {H}^{2}.
\end{cases}
\end{equation}
In view of \eqref{WW1}, \eqref{XXhjiooldd}, \eqref{hjiooldd}  and \eqref{xddhj897}, we obtain
\begin{align}
&\|\partial u(t)-\partial \overline{u}(t)\|_{H^1}\nonumber\\
&\leq \|\partial w(t)-\partial \overline{w}(t)\|_{H^1}
+\|\partial v(t)-\partial \overline{v}(t)\|_{H^1}\longrightarrow 0,~~~\text{as}~~t\longrightarrow \infty.
\end{align}
Thus we get that $u$ scatters.

\subsection{Proof of Proposition \ref{mingti2}}\label{sdfffffff}

In this section, we will prove Proposition \ref{mingti2}.
First, we will show
\begin{Lemma}\label{HDDJU89}
We have
\begin{align}\label{ffhfjkl}
\|\nabla^{\alpha}G(t,\cdot)\|_{L^2}\leq C\langle t\rangle^{-3/2}\varepsilon^2,~~|\alpha|\leq 1.
\end{align}
\end{Lemma}
\begin{proof}
Recalling \eqref{nqsddd33}, we get that for $|\alpha|\leq 1,$
\begin{align}
\big|\nabla^{\alpha}N_2(u,u)\big|&\leq C|  \nabla\nabla\cdot \nabla^{\alpha}u||\nabla\wedge u|
+ C|\nabla\cdot u||\nabla\wedge (\nabla\wedge \nabla^{\alpha}u)|\nonumber\\
&+C|  \nabla\nabla\cdot u||\nabla\wedge \nabla^{\alpha}u|
+ C|\nabla\cdot \nabla^{\alpha}u||\nabla\wedge (\nabla\wedge u)|.
\end{align}
In the region ${r\leq \frac{\langle c_2t\rangle}{2}}$, it holds that $\langle t\rangle\leq C\langle c_1t-r\rangle$, $\langle t\rangle\leq C\langle c_2t-r\rangle $. Thus by \eqref{gao300} and \eqref{Lowenergy},  we have
\begin{align}\label{320}
&\big\|\nabla^{\alpha}N_2(u,u)\big\|_{L^2({r\leq \frac{\langle c_2t\rangle}{2}})}\leq C\|\nabla u\nabla^2 \nabla^{\alpha}u\|_{L^2({r\leq \frac{\langle c_2t\rangle}{2}})}+C\|\nabla \nabla^{\alpha}u\nabla^2 u\|_{L^2({r\leq \frac{\langle c_2t\rangle}{2}})}\nonumber\\
&\leq C\langle t\rangle^{-2}\big(\|\langle c_1t-r\rangle \nabla u_{cf}\|_{L^{\infty}}+\|\langle c_2t-r\rangle \nabla u_{df}\|_{L^{\infty}}\big)\nonumber\\
& ~~~~~~~~~~~~~\cdot
\big(\|\langle c_1t-r\rangle \nabla^2 \nabla^{\alpha}u_{cf}\|_{L^2}+\|\langle c_2t-r\rangle \nabla^2 \nabla^{\alpha}u_{df}\|_{L^2}\big)\nonumber\\
&+ C\langle t\rangle^{-2}\big(\|\langle c_1t-r\rangle \nabla \nabla^{\alpha}u_{cf}\|_{L^{\infty}}+\|\langle c_2t-r\rangle \nabla \nabla^{\alpha}u_{df}\|_{L^{\infty}}\big)\nonumber\\
& ~~~~~~~~~~~~~\cdot
\big(\|\langle c_1t-r\rangle \nabla^2 u_{cf}\|_{L^2}+\|\langle c_2t-r\rangle \nabla^2 u_{df}\|_{L^2}\big)\nonumber\\
&\leq C\langle t\rangle^{-2}\big(E_{4}^{1/2}(u(t))+\mathcal {X}_{4}(u(t))\big)\mathcal {X}_{3}(u(t))\leq C\langle t\rangle^{-2}E_{4}(u(t)).
\end{align}
In the region ${ \frac{\langle c_2t\rangle}{2}\leq r\leq \frac{\langle (c_1+c_2)t\rangle}{2}}$,  $\langle t\rangle\leq C\langle r\rangle$, $\langle t\rangle\leq C\langle c_1t-r\rangle $. From \eqref{gao2}, \eqref{gao300} and \eqref{Lowenergy}, we get
\begin{align}\label{320ii}
&\big\|\nabla^{\alpha}N_2(u,u)\big\|_{L^2({ \frac{\langle c_2t\rangle}{2}\leq r\leq \frac{\langle (c_1+c_2)t\rangle}{2}})}\nonumber\\
&\leq C\langle t\rangle^{-3/2}\|\langle c_1t-r\rangle  \nabla\nabla\cdot  \nabla^{\alpha}u_{cf}\|_{L^2}\|r \nabla  u\|_{L^{\infty}}
+ C\| r^{1/2} \langle c_1t-r\rangle\nabla\cdot u_{cf}\|_{L^{\infty}}\|\nabla^2\nabla^{\alpha}u\|_{L^2}\nonumber\\
&+ C\langle t\rangle^{-3/2}\|\langle c_1t-r\rangle  \nabla\nabla\cdot  u_{cf}\|_{L^2}\|r \nabla  \nabla^{\alpha}u\|_{L^{\infty}}
+ C\| r^{1/2} \langle c_1t-r\rangle\nabla\cdot \nabla^{\alpha}u_{cf}\|_{L^{\infty}}\|\nabla^2u\|_{L^2}\nonumber\\
&\leq C\langle t\rangle^{-3/2}\|\langle c_1t-r\rangle  \nabla^2 \nabla^{\alpha}u_{cf}\|_{L^2}\|r \nabla  u\|_{L^{\infty}}
+ C\| r^{1/2} \langle c_1t-r\rangle\nabla u_{cf}\|_{L^{\infty}}\|\nabla^2\nabla^{\alpha}u\|_{L^2}\nonumber\\
&+ C\langle t\rangle^{-3/2}\|\langle c_1t-r\rangle  \nabla^2 u_{cf}\|_{L^2}\|r \nabla  \nabla^{\alpha}u\|_{L^{\infty}}
+ C\| r^{1/2} \langle c_1t-r\rangle\nabla \nabla^{\alpha}u_{cf}\|_{L^{\infty}}\|\nabla^2u\|_{L^2}\nonumber\\
&\leq C\langle t\rangle^{-3/2}\big(E_{4}^{1/2}(u(t))+\mathcal {X}_{4}(u(t))\big)E_{4}^{1/2}(u(t))\leq C\langle t\rangle^{-3/2}E_{4}(u(t)).
\end{align}
In the region ${  r\geq \frac{\langle (c_1+c_2)t\rangle}{2}}$,  $\langle t\rangle\leq C\langle r\rangle$, $\langle t\rangle\leq C\langle c_2t-r\rangle $. By \eqref{gao2}, \eqref{gao300} and \eqref{Lowenergy},  we also have
\begin{align}\label{320ihhhi}
&\big\|\nabla^{\alpha}N_2(u,u)\big\|_{L^2({  r\geq \frac{\langle (c_1+c_2)t\rangle}{2}})}\nonumber\\
&\leq C\langle t\rangle^{-3/2}\|  \nabla^2 \nabla^{\alpha}u\|_{L^2}\|r^{1/2} \langle c_2t-r\rangle\nabla\wedge   u_{df}\|_{L^{\infty}}
+ C\|r\nabla u\|_{L^{\infty}}\| \langle c_2t-r\rangle\nabla\wedge (\nabla\wedge \nabla^{\alpha}u_{df})\|_{L^2}\nonumber\\
&+C\langle t\rangle^{-3/2}\|  \nabla^2 u\|_{L^2}\|r^{1/2} \langle c_2t-r\rangle\nabla\wedge   \nabla^{\alpha}u_{df}\|_{L^{\infty}}
+ C\|r\nabla \nabla^{\alpha}u\|_{L^{\infty}}\| \langle c_2t-r\rangle\nabla\wedge (\nabla\wedge u_{df})\|_{L^2}\nonumber\\
&\leq C\langle t\rangle^{-3/2}\|  \nabla^2 \nabla^{\alpha}u\|_{L^2}\|r^{1/2} \langle c_2t-r\rangle\nabla  u_{df}\|_{L^{\infty}}
+ C\|r\nabla u\|_{L^{\infty}}\| \langle c_2t-r\rangle\nabla^2\nabla^{\alpha}u_{df}\|_{L^2}\nonumber\\
&+C\langle t\rangle^{-3/2}\|  \nabla^2 u\|_{L^2}\|r^{1/2} \langle c_2t-r\rangle\nabla  \nabla^{\alpha}u_{df}\|_{L^{\infty}}
+ C\|r\nabla \nabla^{\alpha}u\|_{L^{\infty}}\| \langle c_2t-r\rangle\nabla^2u_{df}\|_{L^2}\nonumber\\
&\leq C\langle t\rangle^{-3/2}\big(E_{4}^{1/2}(u(t))+\mathcal {X}_{4}(u(t))\big)E_{4}^{1/2}(u(t))\leq C\langle t\rangle^{-3/2}E_{4}(u(t)).
\end{align}
Thus it follows from \eqref{320}, \eqref{320ii} and \eqref{320ihhhi} that
\begin{align}\label{NN2}
\|\nabla^{\alpha}N_2(u,u)\|_{L^2}\leq C\langle t\rangle^{-3/2}\varepsilon^2.
\end{align}

Similarly to \eqref{320}, from \eqref{gao300} and \eqref{Lowenergy} we have
\begin{align}\label{32XX0}
&\big\|\nabla^{\alpha}N_3(u,u)\big\|_{L^2({r\leq \frac{\langle c_2t\rangle}{2}})}\leq C\|\nabla u\nabla^2 \nabla^{\alpha}u\|_{L^2({r\leq \frac{\langle c_2t\rangle}{2}})}
+C\|\nabla \nabla^{\alpha}u\nabla^2 u\|_{L^2({r\leq \frac{\langle c_2t\rangle}{2}})}\nonumber\\
&\leq C\langle t\rangle^{-2}\big(\|\langle c_1t-r\rangle \nabla u_{cf}\|_{L^{\infty}}+\|\langle c_2t-r\rangle\nabla u_{df}\|_{L^{\infty}}\big)\nonumber\\
& ~~~~~~~~~~~~~\cdot
\big(\|\langle c_1t-r\rangle \nabla^2 \nabla^{\alpha}u_{cf}\|_{L^2}+\|\langle c_2t-r\rangle \nabla^2 \nabla^{\alpha}u_{df}\|_{L^2}\big)\nonumber\\
&+C\langle t\rangle^{-2}\big(\|\langle c_1t-r\rangle \nabla \nabla^{\alpha}u_{cf}\|_{L^{\infty}}+\|\langle c_2t-r\rangle\nabla \nabla^{\alpha}u_{df}\|_{L^{\infty}}\big)\nonumber\\
& ~~~~~~~~~~~~~\cdot
\big(\|\langle c_1t-r\rangle \nabla^2 u_{cf}\|_{L^2}+\|\langle c_2t-r\rangle \nabla^2 u_{df}\|_{L^2}\big)\nonumber\\
&\leq C\langle t\rangle^{-2}\big(E_{4}^{1/2}(u(t))+\mathcal {X}_{4}(u(t))\big)\mathcal {X}_{3}(u(t))\nonumber\\
&\leq C\langle t\rangle^{-2}E_{4}(u(t)).
\end{align}
By Lemma \ref{decay}, \eqref{gao1} and \eqref{gao2} we get
\begin{align}\label{32XssX0}
&\big\|\nabla^{\alpha}N_3(u,u)\big\|_{L^2({r\geq \frac{\langle c_2t\rangle}{2}})}\nonumber\\
&\leq C\langle t\rangle^{-3/2}\sum_{|a|\leq 1}\big(\|r^{1/2}\widetilde{\Omega}^{a}u\|_{L^{\infty}}\|\nabla^2\nabla^{\alpha}u\|_{L^2}
+\|\nabla\widetilde{\Omega}^{a}\nabla^{\alpha}u\|_{L^2}\|r\nabla u\|_{L^{\infty}}\big)\nonumber\\
&+C\langle t\rangle^{-3/2}\sum_{|a|\leq 1}\big(\|r^{1/2}\widetilde{\Omega}^{a}\nabla^{\alpha}u\|_{L^{\infty}}\|\nabla^2u\|_{L^2}
+\|\nabla\widetilde{\Omega}^{a}u\|_{L^2}\|r\nabla \nabla^{\alpha}u\|_{L^{\infty}}\big)\nonumber\\
&\leq C\langle t\rangle^{-3/2}E_{4}(u(t)).
\end{align}
By \eqref{32XX0} and \eqref{32XssX0} we have
\begin{align}\label{HJdd890890}
\|\nabla^{\alpha}N_3(u,u)\|_{L^2}\leq C\langle t\rangle^{-3/2}\varepsilon^2.
\end{align}

Hence the combination of \eqref{NN2} and \eqref{HJdd890890} results in
 \begin{align}\label{ffhfjSSSDkl}
\|\nabla^{\alpha}G(t,\cdot)\|_{L^2}\leq \|\nabla^{\alpha}N_2(u,u)\|_{L^2}+\|\nabla^{\alpha}N_3(u,u)\|_{L^2}\leq C\langle t\rangle^{-3/2}\varepsilon^2.
\end{align}
\end{proof}

Now recall that $w$ satisfies \eqref{gloV2}. Let $w^{(0)}=w, w^{(1)}=w_t, {\bf{w}}=(w^{(0)},w^{(1)})^{T}, {\bf{w_0}}=(u_0,u_1)^{T}, {\bf{G}}=(0,G)^{T}$. Then we have
\begin{equation}\label{systemsinho}
\begin{cases}
{\bf{w}}_t+{\bf{A}}{\bf{w}}={\bf{G}},\\
{\bf{w}}|_{t=0}={\bf{w_0}},
\end{cases}
\end{equation}
where
\begin{align}
{\bf{A}}=
\begin{bmatrix}
0& -I\\
-A&0
\end{bmatrix},
~~A=c_2^2\Delta I+(c_1^2-c_2^2)\nabla\otimes \nabla.
\end{align}

Denote the norm
\begin{equation}
\|{\bf{w}}\|_{E}=\|(w^{(0)},w^{(1)})\|_{\mathcal {H}^2}.
\end{equation}
We have
\begin{Lemma}\label{XUyi900}
There exists ${\bf{\overline{w}}}=(\overline{w}^{(0)},\overline{w}^{(1)})$ such that
\begin{equation}
\|{\bf{w}}(t,\cdot)-{\bf{\overline{w}}}(t,\cdot)\|_{E}\longrightarrow 0,~~~\text{as}~~t\longrightarrow \infty,
\end{equation}
where ${\bf{\overline{w}}}$ satisfies
\begin{equation}
{\bf{\overline{w}}}_t+{\bf{A}}{\bf{\overline{w}}}=0~~~{\rm \text{on}}~~\mathbb{R}^{+}\times \mathbb{R}^3
\end{equation}
for some initial data
\begin{equation}\label{dddddXXsss}
{\bf{\overline{w}}}|_{t=0}={\bf{\overline{w}_0}}=(\overline{w}_0, \overline{w}_1)\in \mathcal {H}^2.
\end{equation}
\end{Lemma}
\begin{proof}
By Duhamel's principle, the solution to the system \eqref{systemsinho} can be expressed as
\begin{align}
{\bf{w}}(t,\cdot)={\bf{S}}(t){\bf{w_0}}+\int_0^{t}{\bf{S}}(t-\tau){\bf{G}}(\tau,\cdot){\rm d}\tau.
\end{align}
Here $\{{\bf{S}}(t): t\in \mathbb{R}\}$ is defined as follows: $\widetilde{{\bf{w}}}={\bf{S}}(t){\bf{\widetilde{w}_0}}$ satisfies
\begin{equation}\label{systemdddSSsinho}
\begin{cases}
{\bf{\widetilde{w}}}_t+{\bf{A}}{\bf{\widetilde{w}}}=0,~~~{\text{on}}~~\mathbb{R}\times \mathbb{R}^3,\\
{\bf{\widetilde{w}}}|_{t=0}={\bf{\widetilde{w}_0}}=(\widetilde{w}_0,\widetilde{w}_1)\in \mathcal {H}^2.
\end{cases}
\end{equation}
We can verify that $\{{\bf{S}}(t): t\in \mathbb{R}\}$ is a group, i.e.,
\begin{align}\label{Group}
{\bf{S}}(t_2+t_1)={\bf{S}}(t_2){\bf{S}}(t_1),~~t_1, t_2\in \mathbb{R}.
\end{align}
 and admits the following estimate
\begin{align}\label{Ugroup}
\|{\bf{S}}(t){\bf{\widetilde{w}_0}}\|_{E}\leq  C\|{\bf{\widetilde{w}_0}}\|_{E},~~t\in \mathbb{R}.
\end{align}
Now let
\begin{align}\label{xksssuo900}
{\bf{\overline{w}_0}}={\bf{w_0}}+\int_0^{\infty}{\bf{S}}(-\tau){\bf{G}}(\tau,\cdot){\rm d}\tau,
\end{align}
which is well defined. In fact, by Minkowski inequality, \eqref{Ugroup} and Lemma \ref{HDDJU89} we have
\begin{align}
\|{\bf{\overline{w}_0}}\|_{E}&\leq \|{\bf{w_0}}\|_{E}+\int_0^{\infty}\|{\bf{S}}(-\tau){\bf{G}}(\tau,\cdot)\|_{E}{\rm d}\tau\nonumber\\
&\leq  C\|{\bf{w_0}}\|_{E}+C\int_0^{\infty}\|{\bf{G}}(\tau,\cdot)\|_{E} {\rm d}\tau\nonumber\\
&\leq C\|(u_0, u_1)\|_{\mathcal {H}^2}+C\int_0^{\infty}\|{{G}}(t,\cdot)\|_{H^1}{\rm d}t\nonumber\\
&\leq  C\varepsilon+C\varepsilon^2\int_0^{\infty}\langle t\rangle^{-3/2}{\rm d}t\leq  C\varepsilon+C\varepsilon^2.
\end{align}
Now take ${\bf{\overline{w}}}=(\overline{w}^{(0)},\overline{w}^{(1)})$ as follows:
\begin{equation}
\begin{cases}
{\bf{\overline{w}}}_t+{\bf{A}}{\bf{\overline{w}}}=0~~~{\text{on}}~~\mathbb{R}^{+}\times \mathbb{R}^3,\\
{\bf{\overline{w}}}|_{t=0}={\bf{\overline{w}_0}}=(\overline{w}_0, \overline{w}_1)\in \mathcal {H}^2.
\end{cases}
\end{equation}
In view of \eqref{xksssuo900} and \eqref{Group}, we have
\begin{align}
{\bf{{w}}}(t,\cdot)-{\bf{\overline{w}}}(t,\cdot)&={\bf{S}}(t){\bf{w_0}}+\int_0^{t}{\bf{S}}(t-\tau){\bf{G}}(\tau,\cdot)d\tau-{\bf{S}}(t){\bf{\overline{w}_0}}\nonumber\\
&=-\int_t^{\infty}{\bf{S}}(t-\tau){\bf{G}}(\tau,\cdot){\rm d}\tau.
\end{align}
Then by Minkowski inequality, \eqref{Ugroup} and Lemma \ref{HDDJU89} we have
\begin{align}
&\|{\bf{{w}}}(t,\cdot)-{\bf{\overline{w}}}(t,\cdot)\|_{E}\leq \int_t^{\infty}\|{\bf{S}}(t-\tau){\bf{G}}(\tau,\cdot)\|_{E}{\rm d}\tau\leq C \int_t^{\infty}\|{\bf{G}}(\tau,\cdot)\|_{E}{\rm d}\tau\nonumber\\
&\leq C\int_t^{\infty}\|{{G}}(\tau,\cdot)\|_{H^1}{\rm d}\tau\leq C\varepsilon^2\int_t^{\infty}\langle \tau\rangle^{-3/2}{\rm d}\tau\leq C\varepsilon^2\langle t\rangle^{-1/2},
\end{align}
which obviously implies
\begin{equation}
\|{\bf{w}}(t,\cdot)-{\bf{\overline{w}}}(t,\cdot)\|_{E}\longrightarrow 0,~~~\text{as}~~t\longrightarrow \infty.
\end{equation}
\end{proof}

Now Proposition \ref{mingti2} is a consequence of Lemma \ref{XUyi900}. In fact, we can just take $\overline{w}=\overline{w}^{(0)}$. Then it follows from Lemma \ref{XUyi900} that $\overline{w}$ satisfies
\begin{equation}
L\overline{w}=0~~~{\text{on}}~~\mathbb{R}^{+}\times \mathbb{R}^3
\end{equation}
with initial data
\begin{equation}
(\overline{w}, \overline{w}_t)|_{t=0}=(\overline{w}_0, \overline{w}_1)\in \mathcal {H}^2.
\end{equation}
And we have
\begin{align}
&\|\partial w(t)-\partial \overline{w}(t)\|^2_{H^1}\nonumber\\
&=\| \nabla w(t)-\nabla \overline{w}(t)\|^2_{{H}^1}+\|\partial_t w(t)-\partial_t \overline{w}(t)\|^2_{H^1}\longrightarrow 0,~~~\text{as}~~t\longrightarrow \infty.
\end{align}
We have completed the proof of Proposition \ref{mingti2}.

\subsection{Proof of Proposition \ref{mingti1}}\label{sdffffdddfff}

Now we turn to the proof of Proposition \ref{mingti1}. Unlike \eqref{ffhfjkl}, now for $|\alpha|\leq 1$ we can only get
\begin{align}
\|\nabla^{\alpha}F(t,\cdot)\|_{L^2}\leq C\langle t\rangle^{-1}\varepsilon^2,
\end{align}
which is not integrable in time. Thus we can not go through as in the previous section. But in view of  \eqref{FF} and \eqref{nqsddd22},
 \begin{align}\label{422}
F(t,x)=N_1(u,u)&=d_2\nabla\big(|\nabla\wedge u|^2\big),
 \end{align}
which admits the gradient structure. Then we know that $v$, the global solution to \eqref{gloV}, is curl-free. In fact, if follows from \eqref{gloV} and \eqref{422} that $\nabla\wedge v$ satisfies
\begin{equation}\label{gloVoo}
\begin{cases}
L(\nabla\wedge v)=0,~~~{\text{on}}~~\mathbb{R}^{+}\times \mathbb{R}^3,\\
(\nabla\wedge v, \nabla\wedge v_t)|_{t=0}=(0, 0),
\end{cases}
\end{equation}
which implies
\begin{equation}\label{dhwwjio}
\nabla\wedge v=0.
\end{equation}
By the Helmholtz decomposition we get
\begin{equation}\label{djo900}
\Delta v=\nabla\nabla\cdot v-\nabla\wedge(\nabla\wedge v).
\end{equation}
Then the combination of \eqref{dhwwjio} and \eqref{djo900} gives
\begin{align}
Lv&=\partial_t^2v-c_2^2\Delta v-(c_1^2-c_2^2)\nabla\nabla \cdot v
=\partial_t^2v-c_1^2\nabla\nabla \cdot v+c_2^2\nabla\wedge(\nabla\wedge v)\nonumber\\
&=
\partial_t^2v-c_1^2\big(\nabla\nabla \cdot v-\nabla\wedge(\nabla\wedge v)\big)=\partial_t^2v-c_1^2\Delta v=\Box_{c_1}v.
\end{align}
Thus, $v$ satisfies the following system of wave equations with speed $c_1$,
\begin{equation}\label{glssoV}
\begin{cases}
\Box_{c_1}v=F(t,x),~~~{\text{on}}~~\mathbb{R}^{+}\times \mathbb{R}^3,\\
(v, v_t)|_{t=0}=(0, 0).
\end{cases}
\end{equation}

In order to prove that $v$ is asymptotically free, we can rely on some tools in the scattering theory of wave equations. For the following related results, we refer the reader to Friedlander \cite{MR142888, MR583989}, Lax and Phillips \cite{MR1037774}, H\"{o}rmander \cite{MR1065993, MR1466700},  Katayama \cite{MR3045636}, etc.

Consider the scalar linear wave equation
\begin{align}\label{jiejki}
\begin{cases}
\Box_{c_1}\psi=0,~~~{\text{on}}~~\mathbb{R}^{+}\times \mathbb{R}^3,\\
(\psi,\psi_{t})|_{t=0}=(\psi_0,\psi_1).
\end{cases}
\end{align}
If $(\psi_0,\psi_1)\in C_c^{\infty}(\mathbb{R}^3)\times C_c^{\infty}(\mathbb{R}^3)$, then $\psi\in C^{\infty}(\mathbb{R}^{+}\times \mathbb{R}^3)$, and for $\sigma\in \mathbb{R}, \omega\in \mathbb{S}^2$, the limit
\begin{equation}
\lim_{r\rightarrow \infty}r\psi(c_1^{-1}(r-\sigma),r\omega)
\end{equation}
exists. We denote it by $F_0[\varphi,\psi](\sigma,\omega)$, i.e.,
\begin{equation}
F_0[\varphi,\psi](\sigma,\omega)=\lim_{r\rightarrow \infty}r\psi(c_1^{-1}(r-\sigma),r\omega).
\end{equation}
We point out that $F_0[\varphi,\psi]$ is just Friedlander's radiation field \cite{MR142888}.
We also have that $F_0[\varphi,\psi]\in C^{\infty}(\mathbb{R}\times \mathbb{S}^2)\bigcap L^2(\mathbb{S}^2; H^2(\mathbb{R}))$, and there is a positive constant $C_1$ such that \begin{equation}\label{From22}
C_1^{-1}\|(\psi_0,\psi_1)\|_{\mathcal {H}^2}\leq \big\|\partial_{\sigma}F_0[\psi_0,\psi_1]\big\|_{L^2(\mathbb{S}^2; H^1(\mathbb{R}))}\leq C_1\|(\psi_0,\psi_1)\|_{\mathcal {H}^2}.
\end{equation}
Furthermore, \begin{align}
 &\big\| \partial \psi(t,x)-\overrightarrow{\omega} r^{-1}\big(\partial_{\sigma}F_0[\psi_0,\psi_1]\big)(r-c_1t,|x|^{-1}{x})\big \|_{L^2}\nonumber\\
 &
 +\big\|\partial^2 \psi(t,x)-(\overrightarrow{\omega}\otimes\overrightarrow{\omega})r^{-1}\big(\partial^2_{\sigma}F_0[\psi_0,\psi_1]\big)(r-c_1t,|x|^{-1}{x})\big\|_{L^2}\longrightarrow 0,~~{\text{as}}~t\longrightarrow \infty,
 \end{align}
where $\overrightarrow{\omega}=(-c_1,|x|^{-1}{x})$.

   It follows from \eqref{From22} that the linear mapping
    \begin{align}
 C_c^{\infty}(\mathbb{R}^3)\times C_c^{\infty}(\mathbb{R}^3)\ni(\psi_0,\psi_1)\longmapsto \partial_{\sigma}F_0[\psi_0,\psi_1]\in L^2(\mathbb{S}^2; H^1(\mathbb{R}))
 \end{align}
can be uniquely extended to a linear mapping $\mathcal {T}$ from  $\mathcal {H}^2$ to $L^2(\mathbb{S}^2; H^1(\mathbb{R}))$, such that
   \begin{equation}\label{Frohhh}
C_1^{-1}\|(\psi_0,\psi_1)\|_{\mathcal {H}^2}\leq \big\|\mathcal {T}[\psi_0,\psi_1]\big\|_{L^2(\mathbb{S}^2; H^1(\mathbb{R}))}\leq C_1\|(\psi_0,\psi_1)\|_{\mathcal {H}^2}
\end{equation}
   holds for $(\psi_0,\psi_1)\in \mathcal {H}^2$. We note that the mapping $\mathcal {T}$ is essentially just the translation representation in Lax and Phillips \cite{MR1037774}.
 Furthermore,
  \begin{align}
 &\big\| \partial \psi(t,x)-\overrightarrow{\omega} r^{-1}\mathcal {T}[\psi_0,\psi_1](r-c_1t,|x|^{-1}{x})\big \|_{L^2}\nonumber\\
 &
 +\big\|\partial^2 \psi(t,x)-(\overrightarrow{\omega}\otimes\overrightarrow{\omega})r^{-1}\big(\partial_{\sigma}\mathcal {T}[\psi_0,\psi_1]\big)(r-c_1t,|x|^{-1}{x})\big\|_{L^2}\longrightarrow 0,~~{\text{as}}~t\longrightarrow \infty,
 \end{align}
 where $(\psi_0,\psi_1)\in \mathcal {H}^2$,  $\psi$ satisfies \eqref{jiejki}. And the mapping $\mathcal {T}$
 is a bijection.


 Summarily, we have the following
\begin{Lemma}\label{xuyaio900}
Let $(\psi_0,\psi_1)\in \mathcal {H}^2$ and $\psi$ be the solution to \eqref{jiejki}. Then there exists a function $G=G(\sigma, \omega)\in L^2(\mathbb{S}^2; H^1(\mathbb{R}))$ such that
\begin{align}
 &\big\| \partial \psi(t,x)-\overrightarrow{\omega} r^{-1}G(r-c_1t,|x|^{-1}{x})\big \|_{L^2}\nonumber\\
 &
 +\big\|\partial^2 \psi(t,x)-(\overrightarrow{\omega}\otimes\overrightarrow{\omega})r^{-1}(\partial_{\sigma}G)(r-c_1t,|x|^{-1}{x})\big\|_{L^2}\longrightarrow 0,~~{\text{as}}~t\longrightarrow \infty.
 \end{align}
  Furthermore, the map
 \begin{align}
 \mathcal {T}:  \mathcal {H}^2\ni(\psi_0,\psi_1)\longmapsto G\in L^2(\mathbb{S}^2; H^1(\mathbb{R}))
 \end{align}
 is a bijection.
\end{Lemma}

Now we have the following
 \begin{Lemma}\label{LEMMAYJK}
 $v$ is asymptotically free, if there exists a function $\Lambda=(\Lambda_1,\Lambda_2,\Lambda_3)$ with $\Lambda_k=\Lambda_k(\sigma,\omega)\in L^2( \mathbb{S}^2; H^1(\mathbb{R}) )$, such that
 \begin{align}\label{xvbfffhjyuiojj}
 &\sum_{k=1}^{3}\big\| \partial v_k(t,x)-\overrightarrow{\omega} r^{-1}\Lambda_k(r-c_1t,|x|^{-1}{x})\big \|_{L^2}\nonumber\\
 &
 +\sum_{k=1}^{3}\big\|\partial^2 v_k(t,x)-(\overrightarrow{\omega}\otimes\overrightarrow{\omega})r^{-1}(\partial_{\sigma}\Lambda_k)(r-c_1t,|x|^{-1}{x})\big\|_{L^2}\longrightarrow 0,~~{\text{as}}~t\longrightarrow \infty.
 \end{align}
 \end{Lemma}
\begin{proof}
For $k=1,2,3$, noting that $\Lambda_k\in L^2( \mathbb{S}^2; H^1(\mathbb{R}) )$, by Lemma \ref{xuyaio900}, we can take
\begin{align}
((\widetilde{{v}}_0)_k, (\widetilde{{v}}_1)_k)=\mathcal {T}^{-1}\Lambda_k\in \mathcal {H}^{2},
\end{align}
such that $\widetilde{{v}}_k$, the solution to linear wave equation
\begin{align}\label{jidddejki}
\begin{cases}
\Box_{c_1}\widetilde{{v}}_k=0,~~~{\text{on}}~~\mathbb{R}^{+}\times \mathbb{R}^3,\\
(\widetilde{{v}}_k,\partial_t\widetilde{{v}}_k)|_{t=0}=((\widetilde{{v}}_0)_k, (\widetilde{{v}}_1)_k),
\end{cases}
\end{align}
satisfies
\begin{align}\label{BHYUO9}
 &\big\| \partial \widetilde{{v}}_k(t,x)-\overrightarrow{\omega} r^{-1}\Lambda_k(r-c_1t,|x|^{-1}{x})\big \|_{L^2}\nonumber\\
 &
 +\big\|\partial^2 \widetilde{{v}}_k(t,x)-(\overrightarrow{\omega}\otimes\overrightarrow{\omega})r^{-1}(\partial_{\sigma}\Lambda_k)(r-c_1t,|x|^{-1}{x})\big\|_{L^2}\longrightarrow 0,~~{\text{as}}~t\longrightarrow \infty.
 \end{align}

 By \eqref{xvbfffhjyuiojj} and \eqref{BHYUO9}, we obtain
 \begin{align}\label{BHYUObb9}
 &\big\| \partial {{v}}_k(t,x)-\partial \widetilde{{v}}_k(t,x)\big \|_{L^2}+\big\| \partial^2 {{v}}_k(t,x)-\partial^2 \widetilde{{v}}_k(t,x)\big \|_{L^2}\nonumber\\
  &\leq \big\| \partial {{v}}_k(t,x)-\overrightarrow{\omega} r^{-1}\Lambda_k(r-c_1t,|x|^{-1}{x})\big \|_{L^2}\nonumber\\
 &
 +\big\|\partial^2 {{v}}_k(t,x)-(\overrightarrow{\omega}\otimes\overrightarrow{\omega})r^{-1}(\partial_{\sigma}\Lambda_k)(r-c_1t,|x|^{-1}{x})\big\|_{L^2}\nonumber\\
  &+\big\|\partial \widetilde{v_k}-\overrightarrow{\omega} r^{-1}\Lambda_k(r-c_1t,|x|^{-1}{x})\big \|_{L^2}\nonumber\\
 &
 +\big\|\partial^2 \widetilde{v_k}(t,x)-(\overrightarrow{\omega}\otimes\overrightarrow{\omega})r^{-1}(\partial_{\sigma}\Lambda_k)(r-c_1t,|x|^{-1}{x})\big\|_{L^2}\longrightarrow 0,~~{\text{as}}~t\longrightarrow \infty.
 \end{align}
 Set $ \widetilde{v}=(\widetilde{v_1},\widetilde{v_2},\widetilde{v_3})$.
Then
we have
\begin{equation}\label{hjffffiooldd}
\|\partial v(t)-\partial \widetilde{{v}}(t)\|_{H^1}\longrightarrow 0,~~~\text{as}~~t\longrightarrow \infty,
\end{equation}
where $\widetilde{{v}}$ satisfies
\begin{equation}\label{udddohjk9890}
\Box_{c_1}\widetilde{{v}}=0~~~{\text{on}}~~\mathbb{R}^{+}\times \mathbb{R}^3
\end{equation}
for some initial data
\begin{equation}\label{dddhhdfffdsss}
(\widetilde{{v}},\widetilde{{v}}_t)|_{t=0}=(\widetilde{{v}}_0, \widetilde{{v}}_1)\in \mathcal {H}^{2}.
\end{equation}

 Now let $\overline{v}$ be the curl-free part of $\widetilde{{v}}$, i.e.,
 \begin{equation}\label{xhki989pss}
 \overline{v}=\widetilde{{v}}_{cf}.
 \end{equation}
Then the combination of \eqref{udddohjk9890} and \eqref{xhki989pss} implies
\begin{equation}
L\overline{v}=\Box_{c_1}\overline{v}=0~~~{\text{on}}~~\mathbb{R}^{+}\times \mathbb{R}^3.
\end{equation}
By \eqref{hodge4567}, noting that $v$ is curl-free, i.e., $v_{df}=0$, we get
\begin{align}\label{erf5uuuuuu}
&\|\partial v(t)-\partial \widetilde{v}(t)\|^2_{H^1}\nonumber\\
&=\|\partial v_{cf}(t)-\partial \widetilde{v}_{cf}(t)\|^2_{H^1}+\|\partial v_{df}(t)-\partial \widetilde{v}_{df}(t)\|^2_{H^1}\nonumber\\
&=\|\partial v(t)-\partial \overline{v}(t)\|^2_{H^1}+\|\partial \widetilde{v}_{df}(t)\|^2_{H^1}.
\end{align}
 It follows from \eqref{hjffffiooldd} and \eqref{erf5uuuuuu} that
 \begin{equation}
\|\partial v(t)-\partial \overline{{{v}}}(t)\|_{H^1}\longrightarrow 0,~~~\text{as}~~t\longrightarrow \infty.
\end{equation}
Set
\begin{equation}
(\overline{v}_0, \overline{v}_1)=((\widetilde{{v}}_0)_{cf}, (\widetilde{{v}}_1)_{cf}).
\end{equation}
Finally, we have
\begin{equation}
\|\partial v(t)-\partial \overline{v}(t)\|_{H^1}\longrightarrow 0,~~~\text{as}~~t\longrightarrow \infty,
\end{equation}
where $\overline{v}$ satisfies
\begin{equation}
L\overline{v}=0~~~{\text{on}}~~\mathbb{R}^{+}\times \mathbb{R}^3
\end{equation}
for initial data
\begin{equation}
(\overline{v}, \overline{v}_t)|_{t=0}=(\overline{v}_0, \overline{v}_1)\in \mathcal {H}^{2}.
\end{equation}
That is,  $v$ is asymptotically free.

\end{proof}

In order to verify the condition \eqref{xvbfffhjyuiojj} in Lemma \ref{LEMMAYJK}, we give some pointwise estimates. The first one is
\begin{Lemma}\label{LEMMAfffYJK}
Let
\begin{align}\label{hki0900sss}
L=\partial_t+c_1\partial_r,~~\underline{L}=\partial_t-c_1\partial_r.
\end{align}
For any scalar function $f$ and $r\geq 1$, we have
\begin{equation}\label{im908790}
\big| r\partial f+\overrightarrow{\omega}(2c_1)^{-1}{\underline{L}}(rf)\big|\leq C\big(|\Gamma f|+|f|+\langle c_1t-r\rangle|\partial f|\big)
\end{equation}
and
\begin{equation}\label{im90879022}
\big| r\partial^2 f+(\overrightarrow{\omega}\otimes \overrightarrow{\omega})(2c_1)^{-1}\partial_r{\underline{L}}(rf)\big|\leq C\big(|\Gamma f|+|\partial f|+|\partial \Gamma f|+\langle c_1t-r\rangle(|\partial f|+|\partial^2 f|)\big).
\end{equation}

 \end{Lemma}
\begin{proof}
Denote the good derivatives (see \cite{Alinhac01}) by
\begin{align}\label{hkUYII}
T=(T_0,T_1,T_2,T_3)=(c_1)^{-1}\overrightarrow{\omega}\partial_t+\partial.
\end{align}
Note that $T_0=0$. While for $i=1,2,3$, by the radial-angular decomposition
\begin{align}\label{radiall}
\nabla=\omega\partial_r-\frac{\omega\wedge \Omega}{r},
\end{align}
we get
\begin{align}\label{imp1}
c_1rT_i=c_1\omega_iS-c_1(\omega\wedge \Omega)_i+\omega_i(r-c_1t)\partial_t.
\end{align}
We also have
\begin{align}
L=\omega_i(\omega_i\partial_t+c_1\partial_i)=c_1\omega_iT_i.
\end{align}
Thus we obtain
\begin{align}\label{ieddemp1}
|Lf|\leq C|Tf|&\leq C\langle r\rangle^{-1}\big(|\Gamma f|+|f|+\langle c_1t-r\rangle|\partial f|\big).
\end{align}
It follows from \eqref{hki0900sss} and \eqref{hkUYII} that
\begin{align}\label{hkUddYII}
T=(2c_1)^{-1}\overrightarrow{\omega}(L+\underline{L})+\partial,
\end{align}
which implies
\begin{align}\label{hggkUddYII}
\partial+(2c_1)^{-1}\overrightarrow{\omega}\underline{L}=T-(2c_1)^{-1}\overrightarrow{\omega}L.
\end{align}
By \eqref{ieddemp1} and \eqref{hggkUddYII}, we get
\begin{equation}\label{xuio99}
\big| \partial f+\overrightarrow{\omega}(2c_1)^{-1}{\underline{L}}f\big|\leq C\langle r\rangle^{-1}\big(|\Gamma f|+|f|+\langle c_1t-r\rangle|\partial f|\big),
\end{equation}
which implies \eqref{im908790}.

By the radial-angular decomposition
\eqref{radiall} we get that for $k=1,2,3$,
\begin{align}\label{xhsssj89}
\partial \partial_k f+\overrightarrow{\omega}\omega_k(2c_1)^{-1}\partial_r{\underline{L}}f=
\partial (\omega_k \partial_rf)+\overrightarrow{\omega}(2c_1)^{-1}{\underline{L}}(\omega_k \partial_rf)-\partial(r^{-1}{(\omega\wedge \Omega)_k}f).
\end{align}
Then by \eqref{xhsssj89}, \eqref{xuio99} and Lemma \ref{comuuio}, we have
\begin{align}\label{xhsssddj89}
&\big|\partial \partial_k f+\overrightarrow{\omega}\omega_k(2c_1)^{-1}\partial_r{\underline{L}}f\big|\nonumber\\
&
\leq \big|
\partial (\omega_k \partial_rf)+\overrightarrow{\omega}(2c_1)^{-1}{\underline{L}}(\omega_k \partial_rf)\big|+\big|\partial(r^{-1}{(\omega\wedge \Omega)_k}f)\big|\nonumber\\
&\leq C\langle r\rangle^{-1}\big(|\Gamma (\omega_k \partial_rf)|+|(\omega_k \partial_rf)|+\langle c_1t-r\rangle|\partial (\omega_k \partial_rf)|\big)+\big|\partial(r^{-1}{(\omega\wedge \Omega)_k}f)\big|\nonumber\\
&\leq C\langle r\rangle^{-1}\big(|\Gamma f|+|\nabla f|+|\nabla \Gamma f|+\langle c_1t-r\rangle(|\nabla f|+|\partial \nabla f|)\big).
\end{align}
By \eqref{hkUddYII}, we also have
\begin{align}\label{xhsxxx89}
&\partial \partial_t f-\overrightarrow{\omega}c_1(2c_1)^{-1}\partial_r{\underline{L}}f=
\partial \partial_t f-\frac{\overrightarrow{\omega}}{2}\partial_r{\underline{L}}f \nonumber\\
&
=T\partial_t f-(2c_1)^{-1}\overrightarrow{\omega}L\partial_tf-(2c_1)^{-1}\overrightarrow{\omega}\underline{L}\partial_tf-\frac{\overrightarrow{\omega}}{2}\partial_r{\underline{L}}f\nonumber\\
&
=T\partial_t f-(2c_1)^{-1}\overrightarrow{\omega}L(\partial_tf+\underline{L}f).
\end{align}
It follows from \eqref{ieddemp1}, \eqref{xhsxxx89} and Lemma \ref{comuuio} that
\begin{align}\label{xdddvfddccs89}
&\big|\partial \partial_t f-\overrightarrow{\omega}c_1(2c_1)^{-1}\partial_r{\underline{L}}f\big|
\leq C\langle r\rangle^{-1}\big(|\partial f|+|\partial \Gamma f|+\langle c_1t-r\rangle(|\partial f|+|\partial^2 f|)\big).
\end{align}
The combination of \eqref{xhsssddj89} and \eqref{xdddvfddccs89} gives
\begin{equation}
\big| \partial^2 f+(\overrightarrow{\omega}\otimes \overrightarrow{\omega})(2c_1)^{-1}\partial_r{\underline{L}}f\big|\leq C\langle r\rangle^{-1}\big(|\Gamma f|+|\partial f|+|\partial \Gamma f|+\langle c_1t-r\rangle(|\partial f|+|\partial^2 f|)\big),
\end{equation}
which implies \eqref{im90879022}.
\end{proof}

We point out that the following pointwise decay lemma is inspired by \cite{MR3045636}, but the decay rates in the assumption and conclusion are very different from the corresponding ones in it.

\begin{Lemma}\label{Task}
Recall that $v$ satisfies
\begin{align}
\begin{cases}
\Box_{c_1}v=F(t,x), \\
t=0: v=0, \partial_tv=0.
\end{cases}
\end{align}
We have that if
\begin{align}\label{as1}
\langle r\rangle\langle c_1t-r\rangle^{1/2}|\Gamma^{a}v|&\leq C\varepsilon^2\langle t\rangle^{\delta}, ~|a|\leq 2,\\\label{as2}
\langle r\rangle\langle c_1t-r\rangle|\partial \Gamma^{a}v|&\leq C\varepsilon^2\langle t\rangle^{\delta},~|a|\leq 2,\\\label{as3}
\langle r\rangle^2\langle c_2t-r\rangle^{3/2}|\nabla^{\alpha} F|&\leq C\varepsilon^2,~~~~~|\alpha|\leq 1,
\end{align}
for some small $\delta>0$ (say, $\delta=\frac{1}{1000}$),
then there exists a function $\Lambda=(\Lambda_1,\Lambda_2,\Lambda_3)$ with $\Lambda_k=\Lambda_k(\sigma,\omega)\in L^2( \mathbb{S}^2; H^1(\mathbb{R}) )$, such that for $k=1,2,3$,
\begin{align}\label{wewillpo}
&\big|r\partial v_k(t,x)-\overrightarrow{\omega}\Lambda_k(r-c_1t,|x|^{-1}{x})\big|+\big|r\partial^2 v_k(t,x)-(\overrightarrow{\omega}\otimes\overrightarrow{\omega})(\partial_{\sigma}\Lambda_k)(r-c_1t,|x|^{-1}{x})\big|\nonumber\\
&
\leq C\varepsilon^2\langle t+r\rangle^{-1+2\delta}\langle c_1t-r\rangle^{-\delta}.
\end{align}
\end{Lemma}

\begin{proof}
In the polar coordinates, we have
\begin{align}
r\Box_{c_1}v_k(t,r\omega)=L\underline{L}\big(rv_k(t,r\omega)\big)-c_1^2r^{-1}\Delta_{\omega}v_k(t,r\omega),
\end{align}
for $(t,r,\omega)\in (0,\infty)\times (0,\infty)\times \mathbb{S}^2 $, where $\Delta_{\omega}=\sum_{1\leq i<j\leq 3}\Omega_{ij}^2$ is the Laplace-Beltrami operator on $\mathbb{S}^2$.
Denote
\begin{align}\label{onthi}
\lambda_k(t,r,\omega)=-(2c_1)^{-1}\underline{L}\big(rv_k(t,r\omega)\big)
\end{align}
and
\begin{align}\label{GG}
\overline{F}_k(t,r,\omega)=-(2c_1)^{-1}\big(c_1^2r^{-1}\Delta_{\omega}v_k(t,r\omega)+rF_k(t,r\omega)\big).
\end{align}
Then we have
\begin{align}\label{xuyao1}
L\lambda_k(t,r,\omega)=\overline{F}_k(t,r,\omega).
\end{align}

We first consider the case $r\geq \frac{c_1}{2}t\geq 1$. In this region, we have
\begin{align}
\langle 1+2c_1^{-1}\rangle^{-1}\langle t+r\rangle\leq r\leq \langle r\rangle\leq \langle t+r\rangle.
\end{align}
By \eqref{as1} and \eqref{as2} we get
\begin{align}\label{ji1}
\big| r^{-1}\Delta_{\omega}v_k(t,r\omega)\big|+\big| \partial_r(r^{-1}\Delta_{\omega}v_k(t,r\omega))\big|\leq C\varepsilon^2\langle t+r\rangle^{-2+\delta}\langle c_1t-r\rangle^{-1/2}.
\end{align}
By \eqref{as3} we have
\begin{align}\label{ji3}
\big| rF_k(t,r\omega)\big|+\big| \partial_r(r^{-1}\Delta_{\omega}v_k(t,r\omega))\big|\big| rF_k(t,r\omega)\big|\leq C\varepsilon^2\langle t+r\rangle^{-1}\langle c_2t-r\rangle^{-3/2}.
\end{align}
It follows from \eqref{GG}, \eqref{ji1} and \eqref{ji3} that
\begin{align}\label{ggGG}
&\big|\overline{F}_k(t,r,\omega)\big|+\big|(\partial_r\overline{F}_k)(t,r,\omega)\big|\nonumber\\
&\leq C\varepsilon^2\langle t+r\rangle^{-2+\delta}\langle c_1t-r\rangle^{-1/2}+C\varepsilon^2\langle t+r\rangle^{-1}\langle c_2t-r\rangle^{-3/2}.
\end{align}

For $t>0, \sigma\in \mathbb{R}$, note that $c_1t+\sigma\geq \frac{c_1}{2}t\geq 1$ is equivalent to $t\geq \overline{t}(\sigma)$,
where
\begin{align}\label{inview}
\overline{t}(\sigma)=\max\{-2c_1^{-1}\sigma,2c_1^{-1}\}.
\end{align}
For any $(\sigma,\omega)\in \mathbb{R}\times \mathbb{S}^2$,
in view of \eqref{xuyao1}, we have
\begin{align}\label{sulwen}
\lambda_k(t,c_1t+\sigma,\omega)=\lambda_k\big(\overline{t}(\sigma),c_1\overline{t}(\sigma)+\sigma,\omega\big)
+\int_{\overline{t}(\sigma)}^{t}\overline{F}_k(\tau,c_1\tau+\sigma,\omega){\rm d}\tau,~~t\geq \overline{t}(\sigma).
\end{align}

Now for $(\sigma,\omega)\in \mathbb{R}\times \mathbb{S}^2$, we define
\begin{align}\label{xyaojkl}
\Lambda_k(\sigma,\omega)=\lambda_k\big(\overline{t}
(\sigma),c_1\overline{t}(\sigma)+\sigma,\omega\big)+\int_{\overline{t}(\sigma)}^{\infty}\overline{F}_k(\tau,c_1\tau+\sigma,\omega){\rm d}\tau,
\end{align}
which is well-defined in view of \eqref{ggGG}.

For $r\geq \frac{c_1}{2}t\geq 1$,
by \eqref{sulwen} and \eqref{xyaojkl} , we get
\begin{align}
\lambda_k(t,r,\omega)-\Lambda_k(r-c_1t,\omega)=-\int_{t}^{\infty}\overline{F}_k(\tau,c_1\tau+r-c_1t,\omega){\rm d}\tau,
\end{align}
which implies
\begin{align}
(\partial_r\lambda_k)(t,r,\omega)-(\partial_{\sigma}\Lambda_k)(r-c_1t,\omega)=-\int_{t}^{\infty}(\partial_r\overline{F}_k)(\tau,c_1\tau+r-c_1t,\omega){\rm d}\tau,
\end{align}
In view of \eqref{ggGG}, we have
\begin{align}
&\int_{t}^{\infty}\big(\big|\overline{F}_k(\tau,c_1\tau+r-c_1t,\omega)\big|+\big|(\partial_r\overline{F}_k)(\tau,c_1\tau+r-c_1t,\omega)\big|\big){\rm d}\tau\nonumber\\
&\leq C\varepsilon^2\int_{t}^{\infty}\langle (c_1+1)\tau+r-c_1t\rangle^{-2+\delta}\langle r-c_1t\rangle^{-1/2}{\rm d}\tau\nonumber\\
&+C\varepsilon^2\int_{t}^{\infty}\langle (c_1+1)\tau+r-c_1t\rangle^{-1}\langle (c_1-c_2)\tau+r-c_1t\rangle^{-3/2}{\rm d}\tau\nonumber\\
&\leq C\varepsilon^2\langle t+r\rangle^{-1+\delta}\langle r-c_1t\rangle^{-1/2}+C\langle t+r\rangle^{-1}\nonumber\\
&\leq C\varepsilon^2\langle t+r\rangle^{-1+\delta}\langle r-c_1t\rangle^{-\delta}.
\end{align}
Thus we obtain
\begin{align}\label{xhjio89}
&\big|\lambda_k(t,r,\omega)-\Lambda_k(r-c_1t,\omega)\big|+\big|(\partial_r\lambda_k)(t,r,\omega)-(\partial_{\sigma}\Lambda_k)(r-c_1t,\omega)\big|\nonumber\\
&\leq C\varepsilon^2\langle t+r\rangle^{-1+\delta}\langle r-c_1t\rangle^{-\delta},
\end{align}
which obviously implies
\begin{align}\label{xuyaWEo2222}
&\big|\lambda_k(t,r,\omega)-\Lambda_k(r-c_1t,\omega)\big|+\big|(\partial_r\lambda_k)(t,r,\omega)-(\partial_{\sigma}\Lambda_k)(r-c_1t,\omega)\big|\nonumber\\
&\leq C\varepsilon^2\langle r-c_1t\rangle^{-1}.
\end{align}

 For $r\geq \frac{c_1}{2}t\geq 1$, by \eqref{onthi}, Lemma \ref{LEMMAfffYJK}, \eqref{as1} and \eqref{as2}, we get
\begin{align}\label{hjjjjj5678}
&\big|r\partial v_k-\overrightarrow{\omega}\lambda_k(t,r,\omega)\big|+\big|r\partial^2 v_k-(\overrightarrow{\omega}\otimes \overrightarrow{\omega})(\partial_r\lambda_k)(t,r,\omega)\big|\nonumber\\
&
= \big|r\partial v_k+(2c)^{-1}\overrightarrow{\omega}\underline{L}\big(rv_k(t,r\omega)\big)\big|+\big|r\partial^2 v_k+(2c)^{-1}(\overrightarrow{\omega}\otimes \overrightarrow{\omega})\partial_r\underline{L}\big(rv_k(t,r\omega)\big)\big|\nonumber\\
&
\leq C\varepsilon^2\langle t+r\rangle^{-1+\delta}.
\end{align}
Now it follows from \eqref{xhjio89} and \eqref{hjjjjj5678} that
   \begin{align}
   &\big|r\partial v_k-\overrightarrow{\omega}\Lambda_k(r-c_1t,\omega)\big|+\big|r\partial^2 v_k-(\overrightarrow{\omega}\otimes \overrightarrow{\omega})(\partial_{\sigma}\Lambda_k)(r-c_1t,\omega)\big|\nonumber\\\
   &\leq \big|r\partial v_k-\overrightarrow{\omega}\lambda_k(t,r,\omega)\big|+ \big|\overrightarrow{\omega}(\lambda_k(t,r,\omega)-\Lambda_k(r-c_1t,\omega))\big|\nonumber\\
   &+\big|r\partial^2 v_k-(\overrightarrow{\omega}\otimes \overrightarrow{\omega})\lambda_k(t,r,\omega)\big|+ \big|(\overrightarrow{\omega}\otimes \overrightarrow{\omega})((\partial_r\lambda_k)(t,r,\omega)-(\partial_{\sigma}\Lambda_k)(r-c_1t,\omega))\big|\nonumber\\
   &\leq C\varepsilon^2\langle t+r\rangle^{-1+2\delta}\langle r-c_1t\rangle^{-\delta}.
\end{align}

On the other hand, by \eqref{onthi}, \eqref{as1} and \eqref{as2}, we also get
\begin{align}\label{xddjk89ljk}
&|\lambda_k(t,r,\omega)|+|(\partial_r\lambda_k)(t,r,\omega)|\leq Cr\big|\partial^2 v_k(t,r,\omega)\big| C\langle r\rangle\big|\partial v_k(t,r,\omega)\big|+C\big| v_k(t,r,\omega)\big|\nonumber\\
&\leq C\varepsilon^2\langle r-c_1t\rangle^{-1}\langle t\rangle^{\delta}+C\varepsilon^2\langle t+r\rangle^{-1+\delta}\langle r-c_1t\rangle^{-1/2}\leq C\varepsilon^2\langle r-c_1t\rangle^{-1}\langle t\rangle^{\delta}.
\end{align}
The combination of \eqref{xuyaWEo2222} and \eqref{xddjk89ljk} implies
\begin{align}\label{xuyaWffEo2222}
\big|\Lambda_k(r-c_1t,\omega)\big|+\big|(\partial_{\sigma}\Lambda_k)(r-c_1t,\omega)\big|\leq C\varepsilon^2\langle r-c_1t\rangle^{-1}\langle t\rangle^{\delta},
\end{align}
for $r\geq \frac{c_1}{2}t\geq 1$.

 Now for fixed  $(\sigma,\omega)\in \mathbb{R}\times \mathbb{S}^2$, let $\overline{r}(\sigma)=c_1\overline{t}(\sigma)+\sigma$. In view of \eqref{inview},
 it is obvious that
 \begin{align}
 \overline{t}(\sigma)\leq C\langle \sigma\rangle,~~\overline{r}(\sigma)\geq \frac{c_1}{2}\overline{t}(\sigma)\geq 1.
 \end{align}
By \eqref{xuyaWffEo2222} we have
 \begin{align}
&\big|\Lambda_k(\sigma,\omega)\big|+\big|(\partial_{\sigma}\Lambda_k)(\sigma,\omega)\big|=\big|\Lambda_k\big(\overline{r}(\sigma)-c_1 \overline{t}(\sigma),\omega\big)\big|+\big|(\partial_{\sigma}\Lambda_k)\big(\overline{r}(\sigma)-c_1 \overline{t}(\sigma),\omega\big)\big|\nonumber\\
&\leq C\varepsilon^2\langle \overline{r}(\sigma)-c_1\overline{t}(\sigma)\rangle^{-1}\langle \overline{t}(\sigma)\rangle^{\delta}\leq C\varepsilon^2\langle \sigma\rangle^{-1+\delta}.
\end{align}
Thus we have
 \begin{align}\label{xuyaWfcccfEo2222}
&\big|\Lambda_k(\sigma,\omega)\big|+\big|(\partial_{\sigma}\Lambda_k)(\sigma,\omega)\big|\leq C\varepsilon^2\langle \sigma\rangle^{-1+\delta}, ~~(\sigma,\omega)\in \mathbb{R}\times \mathbb{S}^2,
\end{align}
which implies
$\Lambda_k=\Lambda_k(\sigma,\omega)\in L^2( \mathbb{S}^2; H^1(\mathbb{R}) ), k=1,2,3$.

 Now we will prove \eqref{wewillpo} for $t\leq 2c_1^{-1}$ or $r\leq c_1t/2$. In these regions, we have $\langle r-c_1t\rangle^{-1}\leq C\langle t+r\rangle^{-1}$.
 By \eqref{as2}, we have
 \begin{align}\label{hjio09}
 &\big|r\partial v_k(t,r,\omega)\big|+ \big|r\partial^2 v_k(t,r,\omega)\big|\leq C\varepsilon^2\langle r-c_1t\rangle^{-1}\langle t\rangle^{\delta}\nonumber\\
 &\leq   C\varepsilon^2\langle t+r\rangle^{-1+2\delta}\langle r-c_1t\rangle^{-\delta}.
 \end{align}
 It follows from \eqref{xuyaWfcccfEo2222} that
 \begin{align}\label{xuyaWfcccfddEmmo2222}
&\big|\Lambda_k(r-c_1t,\omega)\big|+\big|(\partial_{\sigma}\Lambda_k)(r-c_1t,\omega)\big|\nonumber\\
&\leq C\varepsilon^2\langle r-c_1t\rangle^{-1+\delta}\leq   C\varepsilon^2\langle t+r\rangle^{-1+2\delta}\langle r-c_1t\rangle^{-\delta}.
\end{align}
From \eqref{hjio09} and \eqref{xuyaWfcccfddEmmo2222}, we get for $t\leq 2c_1^{-1}$ or $r\leq c_1t/2$,
   \begin{align}
   &\big|r\partial v_k-\overrightarrow{\omega}\Lambda_k(r-c_1t,\omega)\big|+\big|r\partial^2 v_k-(\overrightarrow{\omega}\otimes \overrightarrow{\omega})(\partial_{\sigma}\Lambda_k)(r-c_1t,\omega)\big|\nonumber\\
   &\leq C\varepsilon^2\langle t+r\rangle^{-1+2\delta}\langle r-c_1t\rangle^{-\delta}.
\end{align}

\end{proof}

Now our task is to verify the conditions in  Lemma \ref{Task}, in which the gradient structure in \eqref{422} will play a key role again. We first show the following
\begin{Lemma}
Let the scalar function $\phi$ satisfy
\begin{equation}\label{gloVPHI}
\begin{cases}
\Box_{c_1}\phi=d_2|\nabla\wedge u|^2,~~~{\text{on}}~~\mathbb{R}^{+}\times \mathbb{R}^3,\\
(\phi, \phi_t)|_{t=0}=(0, 0).
\end{cases}
\end{equation}
Then we have
\begin{align}\label{as11}
\langle r\rangle\langle c_1t-r\rangle^{1/2}|\nabla \Gamma^{a}\phi|&\leq C\varepsilon^2\langle t\rangle^{\delta}, ~|a|\leq 2,\\\label{as21}
\langle r\rangle\langle c_1t-r\rangle|\partial \nabla \Gamma^{a}\phi|&\leq C\varepsilon^2\langle t\rangle^{\delta}, ~|a|\leq 2.
\end{align}
\end{Lemma}
\begin{proof}
 It follows from \eqref{Sgao3} and \eqref{Sgao4} that
 \begin{align}\label{as1dd1}
\langle r\rangle\langle c_1t-r\rangle^{1/2}|\nabla \Gamma^{a}\phi|&\leq C\big(E^{1/2}_{5}(\phi(t))+C\mathcal {X}_{5}(\phi(t))\big), ~|a|\leq 2,\\\label{asdd21}
\langle r\rangle\langle c_1t-r\rangle|\partial \nabla \Gamma^{a}\phi|&\leq C\mathcal {X}_{6}(\phi(t)), ~~~~~~~~~~~~~~~~~~~~~~~|a|\leq 2.
\end{align}
By the Klainerman-Sideris estimate (see \cite{Klainerman96})
\begin{align}
\|\langle c_1t-r\rangle\partial \nabla \phi(t)\|_{L^2}\leq C\sum_{|a|\leq 1}\|\partial \Gamma^a\phi(t)\|_{L^2}+Ct\|\Box_{c_1}\phi(t)\|_{L^2},
\end{align}
we have
\begin{align}\label{xhioyissSS}
\mathcal {X}_{6}(\phi(t))&\leq CE^{1/2}_{6}(\phi(t))+Ct\sum_{|a|\leq 4}\|\Gamma^{a}\Box_{c_1}\phi(t)\|_{L^2}\nonumber\\
&\leq CE^{1/2}_{6}(\phi(t))+Ct\sum_{|a|\leq 4}\|\Gamma^{a}\big(|\nabla\wedge u|^2\big)\|_{L^2}\nonumber\\
&\leq CE^{1/2}_{6}(\phi(t))+Ct\sum_{|a|\leq 4}\sum_{b+c\leq a}\|\nabla \Gamma^{b}u\nabla \Gamma^{c}u\|_{L^2}.
\end{align}
For the second term on the right hand side of \eqref{xhioyissSS}, for any $|a|\leq 4, b+c\leq a$. it follows from \eqref{gao300} and \eqref{Lowenergy} that
\begin{align}\label{xui9080ssss}
&\|\nabla \Gamma^{b}u\nabla \Gamma^{c}u\|_{L^2({r\leq \frac{\langle c_2t\rangle}{2}})}\nonumber\\
&\leq C\langle t\rangle^{-1}
\begin{cases}
\big(\|\langle c_1t-r\rangle \nabla \Gamma^{b}u_{cf}\|_{L^{\infty}}+\|\langle c_2t-r\rangle \nabla \Gamma^{b}u_{df}\|_{L^{\infty}}\big)\|\nabla \Gamma^{c}u\|_{L^2},~|b|\leq 2\\
\|\nabla \Gamma^{b}u\|_{L^2}\big(\|\langle c_1t-r\rangle \nabla \Gamma^{c}u_{cf}\|_{L^{\infty}}+\|\langle c_2t-r\rangle \nabla \Gamma^{c}u_{df}\|_{L^{\infty}}\big),~|c|\leq 2
\end{cases}\nonumber\\
&\leq C\langle t\rangle^{-1}\big(E^{1/2}_{5}(u(t))+C\mathcal {X}_{5}(u(t))\big)E^{1/2}_{4}(u(t))\nonumber\\
&\leq C\langle t\rangle^{-1}E_{5}(u(t)),
\end{align}
and \eqref{gao2} implies
\begin{align}\label{hjyuidjjjd}
&\|\nabla \Gamma^{b}u\nabla \Gamma^{c}u\|_{L^2({r\geq \frac{\langle c_2t\rangle}{2}})}\nonumber\\
&\leq C\langle t\rangle^{-1}
\begin{cases}
\|r \nabla \Gamma^{b}u\|_{L^{\infty}}\|\nabla \Gamma^{c}u\|_{L^2},~|b|\leq 2\\
\|\nabla \Gamma^{b}u\|_{L^2}\|r \nabla \Gamma^{c}u\|_{L^{\infty}},~|c|\leq 2
\end{cases}\nonumber\\
&
\leq C\langle t\rangle^{-1}E_{5}(u(t)).
\end{align}
Hence it follows from \eqref{xui9080ssss}, \eqref{hjyuidjjjd} and \eqref{Lowenergy} that
\begin{align}\label{hjyuidd}
\sum_{|a|\leq 4}\sum_{b+c\leq a}\|\nabla \Gamma^{b}u\nabla \Gamma^{c}u\|_{L^2}\leq C\langle t\rangle^{-1}E_{5}(u(t))\leq C\langle t\rangle^{-1}\varepsilon^2.
\end{align}
For the first term on the right hand side of \eqref{xhioyissSS}, in view of \eqref{gloVPHI}, the energy approach yields
\begin{align}\label{hjyu7ddd89}
E^{1/2}_{6}(\phi(t))&\leq CE^{1/2}_{6}(\phi(0))+C\sum_{|a|\leq 5}\int_0^{t}\|\Gamma^{a}\Box_{c_1}\phi(\tau)\|_{L^2}{\rm d}\tau\nonumber\\
&\leq C\varepsilon^2+C\sum_{|a|\leq 5}\sum_{b+c\leq a}\int_0^{t}\|\nabla \Gamma^{b}u\nabla \Gamma^{c}u\|_{L^2} {\rm d}\tau.
\end{align}
Similarly to the proof of \eqref{hjyuidd}, we can get
\begin{align}\label{hjyuddidd}
\sum_{|a|\leq 5}\sum_{b+c\leq a}\|\nabla \Gamma^{b}u\nabla \Gamma^{c}u\|_{L^2}\leq C\langle t\rangle^{-1}E_{6}(u(t))\leq C\langle t\rangle^{-1}\varepsilon^2.
\end{align}
The combination of \eqref{hjyu7ddd89} and \eqref{hjyuddidd} implies
\begin{align}\label{xuyoppsssw}
E^{1/2}_{6}(\phi(t))\leq C\varepsilon^2\langle t\rangle^{\delta}.
\end{align}
By \eqref{xhioyissSS}, \eqref{hjyuidd} and \eqref{xuyoppsssw}, we obtain
\begin{align}\label{xhiodddyissddSS}
\mathcal {X}_{6}(\phi(t))\leq C\varepsilon^2\langle t\rangle^{\delta}.
\end{align}
The combination of
\eqref{as1dd1}, \eqref{xuyoppsssw} and \eqref{xhiodddyissddSS} results in \eqref{as11}, and it follows from \eqref{asdd21} and \eqref{xhiodddyissddSS} that  \eqref{as21} holds.
\end{proof}

Now we will verify the conditions \eqref{as1}, \eqref{as2} and \eqref{as3} in Lemma \ref{Task}.
In view of \eqref{422}, \eqref{glssoV} and \eqref{gloVPHI}, we have
\begin{align}\label{nqsdddddd22}
v=\nabla \phi.
\end{align}
Thus \eqref{as11} and \eqref{as21} imply \eqref{as1} and \eqref{as2}, respectively. As for \eqref{as3}, it follows from
\eqref{422}, \eqref{gao3}, \eqref{gao4} and \eqref{Lowenergy}  that for $|\alpha|\leq 1$,
\begin{align}\label{nONLOCAL}
&\langle r\rangle^2\langle c_2t-r\rangle^{3/2}|\nabla^{\alpha}F|\nonumber\\
&\leq C\| \langle r\rangle\langle c_2t-r\rangle \nabla \nabla\wedge \nabla^{\alpha}u\|_{L^{\infty}}\|\langle r\rangle\langle c_2t-r\rangle^{1/2}\nabla\wedge u\|_{L^{\infty}}\nonumber\\&
+C\| \langle r\rangle\langle c_2t-r\rangle \nabla \nabla\wedge u\|_{L^{\infty}}\|\langle r\rangle\langle c_2t-r\rangle^{1/2}\nabla\wedge \nabla^{\alpha}u\|_{L^{\infty}}\nonumber\\
&\leq C\| \langle r\rangle\langle c_2t-r\rangle \nabla \nabla\wedge \nabla^{\alpha}u_{df}\|_{L^{\infty}}\|\langle r\rangle\langle c_2t-r\rangle^{1/2}\nabla\wedge u_{df}\|_{L^{\infty}}\nonumber\\
&
+C\| \langle r\rangle\langle c_2t-r\rangle \nabla \nabla\wedge u_{df}\|_{L^{\infty}}\|\langle r\rangle\langle c_2t-r\rangle^{1/2}\nabla\wedge \nabla^{\alpha}u_{df}\|_{L^{\infty}}\nonumber\\
&\leq  C\| \langle r\rangle\langle c_2t-r\rangle \nabla^2 \nabla^{\alpha}u_{df}\|_{L^{\infty}}\|\langle r\rangle\langle c_2t-r\rangle^{1/2}\nabla u_{df}\|_{L^{\infty}}\nonumber\\
&
+C
\| \langle r\rangle\langle c_2t-r\rangle \nabla^2 u_{df}\|_{L^{\infty}}\|\langle r\rangle\langle c_2t-r\rangle^{1/2}\nabla \nabla^{\alpha}u_{df}\|_{L^{\infty}}\nonumber\\
&
\leq C\big(E_{5}^{1/2}(u(t))+\mathcal {X}_{5}(u(t))\big)\mathcal {X}_{5}(u(t))\leq CE_{5}(u(t))\leq C\varepsilon^2,
\end{align}
which implies \eqref{as3}.

Thus according to Lemma \ref{Task}, we see that there exists a function $\Lambda=(\Lambda_1,\Lambda_2,\Lambda_3)$ with $\Lambda_k=\Lambda_k(\sigma,\omega)\in L^2( \mathbb{S}^2; H^1(\mathbb{R}) )$, such that
\eqref{wewillpo} holds, which implies
 \begin{align}\label{wewillpcco}
 &\big\|\partial v_k(t,x)-\overrightarrow{\omega}r^{-1}\Lambda_k(r-c_1t,|x|^{-1}{x})\big\|_{L^2}^2\nonumber\\
 &
 +\big\|\partial^2 v_k(t,x)-(\overrightarrow{\omega}\otimes\overrightarrow{\omega})r^{-1}(\partial_{\sigma}\Lambda_k)(r-c_1t,|x|^{-1}{x})\big\|_{L^2}^2\nonumber\\
&=\int_0^{\infty}\!\!\!\int_{S^2}\big|r\partial v_k(t,x)-\overrightarrow{\omega}\Lambda_k(r-c_1t,|x|^{-1}{x})\big|^2{\rm d}S_{\omega}{\rm d}r\nonumber\\
&+\int_0^{\infty}\!\!\!\int_{S^2}|r\partial^2 v_k(t,x)-(\overrightarrow{\omega}\otimes\overrightarrow{\omega})(\partial_{\sigma}\Lambda_k)(r-c_1t,|x|^{-1}{x})\big|^2{\rm d}S_{\omega}{\rm d}r\nonumber\\
&\leq C\varepsilon^2\int_0^{\infty}\langle t+r\rangle^{-2+4\delta}\langle c_1t-r\rangle^{-2\delta}{\rm d}r\leq C\varepsilon^2\langle t\rangle^{-1+2\delta}\longrightarrow 0,~~~\text{as}~~t\longrightarrow \infty.
\end{align}
Then by Lemma \ref{LEMMAYJK}, we know that
 $v$ is asymptotically free. We have completed the proof of Proposition \ref{mingti1}.

 \begin{rem}
 We should point out that the Helmholtz decomposition plays a key role in the proof of \eqref{nONLOCAL}. In other places concerning the using of Helmholtz decomposition in this manuscript, the local decomposition in {\rm \cite{Sideris00}} also works. The proof of \eqref{nONLOCAL} is the only place that we must use the nonlocal Helmholtz decomposition, while the local decomposition does not work.
 \end{rem}
\section{Proof of Theorem \ref{mainthm}: rigidity}\label{sbj8978}
\subsection{Energy estimates}
Noting \eqref{Cauchy}--\eqref{xuyaoghujjk}, and the null condition \eqref{null1111},
for $|\alpha|=0, 1$, we have
\begin{align}\label{hjoo67}
L\nabla^{\alpha}u=c_{\alpha}\Big[N_1(u,\nabla^{\alpha}u)+N_2(u,\nabla^{\alpha}u)+N_3(u,\nabla^{\alpha}u)\Big],
\end{align}
where $c_{\alpha}=1$ for $\alpha=0$, and $c_{\alpha}=2$ for $|\alpha|=1$.

Multiplying $\partial_t\nabla^{\alpha}u$ on both sides of \eqref{hjoo67} results in
\begin{align}\label{hjoo}
\langle \partial_t\nabla^{\alpha}u, L\nabla^{\alpha}u\rangle =c_{\alpha}\langle \partial_t\nabla^{\alpha}u,N_1(u,\nabla^{\alpha}u)+N_2(u,\nabla^{\alpha}u)\rangle+c_{\alpha}\langle \partial_t\nabla^{\alpha}u,N_3(u,\nabla^{\alpha}u)\rangle.
\end{align}
By Leibniz's rule we get
\begin{align}\label{xuyaddo344}
\langle \partial_t\nabla^{\alpha}u, L\nabla^{\alpha}u\rangle =\partial_te_1+\nabla\cdot p_1,
\end{align}
where
\begin{align}\label{xuyaouiosss}
e_1&=\frac{1}{2}\Big[|\partial_t\nabla^{\alpha}u|^2+c_2^2|\nabla \nabla^{\alpha}u|^2+(c_1^2-c_2^2)(\nabla \cdot \nabla^{\alpha}u)^2\Big],\\
p_1&=-c_2^2\partial_t \nabla^{\alpha}u\nabla \nabla^{\alpha}u-(c_1^2-c_2^2)\partial_t\nabla^{\alpha}u \nabla\cdot \nabla^{\alpha}u.
\end{align}
Leibniz's rule also implies
\begin{align}\label{fiooddd}
c_{\alpha}\langle \partial_t\nabla^{\alpha}u, N_1(u,\nabla^{\alpha}u)+N_2(u,\nabla^{\alpha}u)\rangle =\partial_te_2+\nabla\cdot p_2+q_2,
\end{align}
where
\begin{align}
e_2&=-c_{\alpha}d_2\langle \nabla\wedge \nabla^{\alpha}u, \frac{1}{2}\nabla\cdot u \nabla\wedge \nabla^{\alpha}u+ \nabla\wedge u \nabla\cdot \nabla^{\alpha}u\rangle,\\
p_2&=c_{\alpha}d_2\Big[\partial_t\nabla^{\alpha}u\big(\nabla \wedge u\cdot \nabla \wedge \nabla^{\alpha}u\big)+\partial_t\nabla^{\alpha}u\wedge \big(\nabla \cdot u\cdot \nabla \wedge \nabla^{\alpha}u
+ \nabla \wedge u  \nabla \cdot\nabla^{\alpha}u\big)\Big],\\
q_2&=c_{\alpha}d_2\langle \nabla\wedge \nabla^{\alpha}u, \frac{1}{2}\partial_t\nabla\cdot u \nabla\wedge \nabla^{\alpha}u+\partial_t\nabla\wedge u \nabla\cdot \nabla^{\alpha}u\rangle,
\end{align}
and
\begin{align}\label{ghui89}
c_{\alpha}\langle \partial_t\nabla^{\alpha}u, N_3(u,\nabla^{\alpha}u)\rangle =\partial_te_3+\nabla\cdot p_3+q_3,
\end{align}
where
\begin{align}
&e_3=c_{\alpha}(d_3+\frac{d_4}{2})\Big[\partial_k\nabla^{\alpha}u^{k}Q_{ij}(u^{j},\nabla^{\alpha}u^{i})-\partial_k\nabla^{\alpha}u^{i}Q_{ij}(u^{k},\nabla^{\alpha}u^{j})\Big]\nonumber\\
&~~~+\frac{c_{\alpha}}{2}d_5\Big[\partial_j\nabla^{\alpha}u^{k}Q_{ij}(u^{k},\nabla^{\alpha}u^{i})-\partial_j\nabla^{\alpha}u^{k}Q_{ij}(u^{i},\nabla^{\alpha}u^{k})
-\partial_ju^{k}Q_{ij}(\nabla^{\alpha}u^{i},\nabla^{\alpha}u^{k})\Big],\\
&(p_3)_i=c_{\alpha}(d_3+\frac{d_4}{2})\Big[2\partial_ju^j\partial_k\nabla^{\alpha}u^{k}\partial_t\nabla^{\alpha}u^{i}-\partial_ju^k\partial_k\nabla^{\alpha}u^{i}
\partial_t\nabla^{\alpha}u^{j}-
\partial_ju^i\partial_k\nabla^{\alpha}u^{j}\partial_t\nabla^{\alpha}u^{k}\Big]\nonumber\\
&~~~~~~+\frac{c_{\alpha}}{2}d_5\Big[\partial_ju^{k}(2\partial_t\nabla^{\alpha}u^{i}\partial_j\nabla^{\alpha}u^{k}-\partial_t\nabla^{\alpha}u^{j}\partial_i\nabla^{\alpha}u^{k}
-\partial_t\nabla^{\alpha}u^{k}\partial_j\nabla^{\alpha}u^{i}-\partial_t\nabla^{\alpha}u^{k}\partial_i\nabla^{\alpha}u^{j})\nonumber\\
&~~~~~~~~~~~~~~~~+\partial_iu^{k}(-\partial_t\nabla^{\alpha}u^{j}\partial_j\nabla^{\alpha}u^{k}+2\partial_t\nabla^{\alpha}u^{k}\partial_j\nabla^{\alpha}u^{j}-
\partial_t\nabla^{\alpha}u^{j}\partial_k\nabla^{\alpha}u^{j})\nonumber\\
&~~~~~~~~~~~~~~~~-\partial_iu^{k}\partial_t\nabla^{\alpha}u^{j}\partial_j\nabla^{\alpha}u^{k}-\partial_ku^{i}\partial_t\nabla^{\alpha}u^{j}\partial_k\nabla^{\alpha}u^{j}
+2\partial_ku^{k}\partial_t\nabla^{\alpha}u^{j}\partial_i\nabla^{\alpha}u^{j}\Big],\\
&q_3=c_{\alpha}(d_3+\frac{d_4}{2})\Big[-\partial_k\nabla^{\alpha}u^{k}Q_{ij}(\partial_tu^{j},\nabla^{\alpha}u^{i})+\partial_k\nabla^{\alpha}u^{i}Q_{ij}
(\partial_tu^{k},\nabla^{\alpha}u^{j})\Big]\nonumber\\
&~~~+\frac{c_{\alpha}}{2}d_5\Big[-\partial_j\nabla^{\alpha}u^{k}Q_{ij}(\partial_tu^{k},\nabla^{\alpha}u^{i})+\partial_j\nabla^{\alpha}u^{k}Q_{ij}(\partial_tu^{i},\nabla^{\alpha}u^{k})
+\partial_t\partial_ju^{k}Q_{ij}(\nabla^{\alpha}u^{i},\nabla^{\alpha}u^{k})\Big].
\end{align}

Set
\begin{align}\label{errro90}
e=e_1-e_2-e_3,~~p=p_1-p_2-p_3,~~q=q_2+q_3.
\end{align}
In view of \eqref{hjoo}, \eqref{xuyaddo344}, \eqref{fiooddd}, \eqref{ghui89} and \eqref{errro90}, we have
\begin{equation}\label{integrjk}
\partial_te+\nabla\cdot p=q.
\end{equation}

Noting the smallness of $|\nabla u|$, we have that there exists a positive constant $c_1>1$ such that
\begin{equation}\label{rggggg}
c_1^{-1}e_1\leq e\leq c_1e_1.
\end{equation}
Obviously we have the rough bound
\begin{equation}\label{roughssui}
|q|\leq |q_2|+|q_3|\leq C|\partial_t\nabla u||\nabla\nabla^{\alpha}u|^2.
\end{equation}
We also have
\begin{align}\label{Q221}
|q_2|&\leq C|\nabla\wedge \nabla^{\alpha}u| |\partial_t\nabla\cdot u_{cf}|| \nabla\wedge \nabla^{\alpha}u|+C|\nabla\wedge \nabla^{\alpha}u| |\partial_t\nabla\wedge u|| \nabla\cdot \nabla^{\alpha}u_{cf}|\nonumber\\
&\leq C|\nabla \nabla^{\alpha}u| |\partial_t\nabla u_{cf}|| \nabla\nabla^{\alpha}u|+C|\nabla \nabla^{\alpha}u| |\partial_t\nabla u|| \nabla \nabla^{\alpha}u_{cf}|,\\\label{Q222}
|q_2|&\leq C|\nabla\wedge \nabla^{\alpha}u_{df}| |\partial_t\nabla\cdot u|| \nabla\wedge \nabla^{\alpha}u|+C|\nabla\wedge \nabla^{\alpha}u_{df}| |\partial_t\nabla\wedge u|| \nabla\cdot \nabla^{\alpha}u|\nonumber\\
&\leq C|\nabla \nabla^{\alpha}u_{df}| |\partial_t\nabla u|| \nabla\nabla^{\alpha}u|,
\end{align}
and
\begin{align}\label{Q3qq}
|q_3|\leq \frac{C}{r}\sum_{|a|\leq 1}\big(|\widetilde{\Omega}^a\partial_tu||\nabla \nabla^{\alpha} u|^2+|\partial_t\nabla u||\widetilde{\Omega}^a\nabla^{\alpha} u||\nabla \nabla^{\alpha} u|\big),
\end{align}
which is implied by Lemma \ref{decay}.

On both sides of \eqref{integrjk}, integrating with respect to spatial variable on $\mathbb{R}^3$, by divergence theorem  we have
\begin{equation}\label{dddrtttt}
\frac{{\rm d}}{{\rm d}t}\int_{\mathbb{R}^3}e(t,x){\rm d}x=\int_{\mathbb{R}^3}q(t,x){\rm d}x.
\end{equation}
In view of \eqref{xhsssk90}, \eqref{xuyaouiosss} and \eqref{rggggg}, we see that
\begin{equation}\label{uccuios}
c_1^{-1}\mathcal {E}_1(\nabla^{\alpha}u(t))\leq \int_{\mathbb{R}^3}e(t,x){\rm d}x\leq c_1\mathcal {E}_1(\nabla^{\alpha}u(t)).
\end{equation}
By \eqref{roughssui}, we get
\begin{equation}\label{gjuip}
\|q(t,\cdot)\|_{L^1(r\leq \frac{\langle c_2t\rangle}{2})}\leq C\|\partial_t\nabla u\|_{L^{\infty}(r\leq \frac{\langle c_2t\rangle}{2})}\|\nabla\nabla^{\alpha}u\|_{L^2}^2.
\end{equation}
In view of \eqref{i22nview1} and \eqref{i22ssnview1}, we have
\begin{align}
\partial_t\nabla u=t^{-1}\big({\widetilde{S}\nabla u-r\partial_r\nabla u+\nabla u}\big).
\end{align}
Consequently, the following pointwise estimate holds
\begin{align}\label{point}
|\partial_t\nabla u|\leq C\langle t\rangle^{-1}(|\nabla \widetilde{S} u|+\langle r\rangle|\partial\nabla u|+|\nabla u|).
\end{align}
Thus by \eqref{gao300}, \eqref{gao4} and \eqref{Lowenergy}, we get
\begin{align}\label{eeeeccf}
&\|\partial_t\nabla u\|_{L^{\infty}(r\leq \frac{\langle c_2t\rangle}{2})}\nonumber\\
&
\leq C\langle t\rangle^{-2}\big(\|\langle c_1t-r\rangle\nabla \widetilde{S} u_{cf}\|_{L^{\infty}}+
\|\langle c_2t-r\rangle\nabla \widetilde{S} u_{df}\|_{L^{\infty}}\big)\nonumber\\
&+C\langle t\rangle^{-2}\big(\|\langle r\rangle\langle c_1t-r\rangle \partial\nabla u_{cf}|_{L^{\infty}}+\|\langle r\rangle\langle c_2t-r\rangle \partial\nabla u_{df}|_{L^{\infty}}\big)\nonumber\\
&
+C\langle t\rangle^{-2}\big(\|\langle c_1t-r\rangle \nabla u_{cf}\|_{L^{\infty}}
+\|\langle c_2t-r\rangle \nabla u_{df}\|_{L^{\infty}}\big)\nonumber\\
&\leq C\langle t\rangle^{-2}\big(E_4^{1/2}(u(t))+C\mathcal {X}_4(u(t))\big)\leq C\langle t\rangle^{-2}E_4^{1/2}(u(t)).
\end{align}
The combination of \eqref{gjuip} and \eqref{eeeeccf} implies
\begin{equation}\label{ggggjuip}
\|q(t,\cdot)\|_{L^1(r\leq \frac{\langle c_2t\rangle}{2})}\leq C\langle t\rangle^{-2}E_4^{1/2}(u(t))\mathcal {E}_2(u(t)).
\end{equation}
It follows from \eqref{Q3qq} and \eqref{gao1} that
 \begin{align}\label{Q3hhhhqq}
&\|q_3(t,\cdot)\|_{L^1(r\geq \frac{\langle c_2t\rangle}{2})}\nonumber\\
&\leq {C}\langle t\rangle^{-3/2}\sum_{|a|\leq 1}\big(\|r^{1/2}\widetilde{\Omega}^a\partial_tu\|_{L^{\infty}}\|\nabla \nabla^{\alpha} u\|_{L^2}^2+\|r^{1/2}\widetilde{\Omega}^a\nabla^{\alpha} u\|_{L^{\infty}}\|\partial_t\nabla u\|_{L^2}\|\nabla \nabla^{\alpha} u\|_{L^2}\big)\nonumber\\
&\leq C\langle t\rangle^{-3/2}E_4^{1/2}(u(t))\mathcal {E}_2(u(t)).
\end{align}
By \eqref{gao3} and \eqref{Q221}, we have
\begin{align}\label{Q3hhffffhhqq}
&\|q_2(t,\cdot)\|_{L^1(\frac{\langle c_2t\rangle}{2}\leq r\leq \frac{\langle (c_1+c_2)t\rangle}{2})}\nonumber\\
&\leq C\langle t\rangle^{-3/2}\|\nabla \nabla^{\alpha}u\|_{L^2} \|\langle r\rangle\langle c_1t-r\rangle^{1/2}\partial_t\nabla u_{cf}\|_{L^{\infty}}\| \nabla\nabla^{\alpha}u\|_{L^2}\nonumber\\
&
+C\langle t\rangle^{-3/2}\|\nabla \nabla^{\alpha}u\|_{L^2} \|\partial_t\nabla u\|_{L^2}\| \langle r\rangle\langle c_1t-r\rangle^{1/2}\nabla \nabla^{\alpha}u_{cf}\|_{L^{\infty}}\nonumber\\
&\leq C\langle t\rangle^{-3/2}\big(E_4^{1/2}(u(t))+C\mathcal {X}_4(u(t))\big)\mathcal {E}_2(u(t))\leq C\langle t\rangle^{-3/2}E_4^{1/2}(u(t))\mathcal {E}_2(u(t)).
\end{align}
Similarly, by \eqref{gao3} and \eqref{Q222}, we also have
\begin{align}\label{Q3hret55hhhqq}
&\|q_2(t,\cdot)\|_{L^1( r\geq \frac{\langle (c_1+c_2)t\rangle}{2})}\nonumber\\
&\leq C\langle t\rangle^{-3/2}\|\langle r\rangle\langle c_2t-r\rangle^{1/2}\nabla \nabla^{\alpha}u_{df}\|_{L^{\infty}} \|\partial_t\nabla u\|_{L^{2}}\| \nabla\nabla^{\alpha}u\|_{L^2}\nonumber\\
&\leq C\langle t\rangle^{-3/2}\big(E_4^{1/2}(u(t))+C\mathcal {X}_4(u(t))\big)\mathcal {E}_2(u(t))\leq C\langle t\rangle^{-3/2}E_4^{1/2}(u(t))\mathcal {E}_2(u(t)).
\end{align}
The combination of \eqref{Q3hhhhqq}, \eqref{Q3hhffffhhqq} and \eqref{Q3hret55hhhqq} follows that
\begin{align}\label{ggggjdddduip}
&\|q(t,\cdot)\|_{L^1(r\geq \frac{\langle c_2t\rangle}{2})}\nonumber\\
&\leq \|q_2(t,\cdot)\|_{L^1(r\geq \frac{\langle c_2t\rangle}{2})}
+\|q_3(t,\cdot)\|_{L^1(r\geq \frac{\langle c_2t\rangle}{2})}\nonumber\\
&\leq C\langle t\rangle^{-3/2}E_4^{1/2}(u(t))\mathcal {E}_2(u(t)).
\end{align}
By \eqref{ggggjuip} and \eqref{ggggjdddduip} we obtain
\begin{equation}\label{ggggjdssssddduip}
\|q(t,\cdot)\|_{L^1}\leq  C\langle t\rangle^{-3/2}E_4^{1/2}(u(t))\mathcal {E}_2(u(t)).
\end{equation}

For $0\leq t\leq t_1<+\infty$, integrating respect to time variable from $t$ to $t_1$ on both sides of
\eqref{dddrtttt}, the fundamental theorem of calculus gives
\begin{equation}\label{dddrtttdddt}
\int_{\mathbb{R}^3}e(t,x){\rm d}x=\int_{\mathbb{R}^3}e(t_1,x){\rm d}x-\int_t^{t_1}\!\!\!\int_{\mathbb{R}^3}q(\tau,x){\rm d}x{\rm d}\tau.
\end{equation}
It follows from \eqref{uccuios} and \eqref{dddrtttdddt} that
\begin{equation}\label{ddjjvvjdrtttdddt}
\mathcal {E}_1(\nabla^{\alpha}u(t))\leq c_1^2\mathcal {E}_1(\nabla^{\alpha}u(t_1))+c_1\int_t^{t_1}\|q(\tau,\cdot)\|_{L^1}{\rm d}\tau,
\end{equation}
for $|\alpha|=0,1$. Thus by \eqref{ddjjvvjdrtttdddt} and \eqref{ggggjdssssddduip}  we have
\begin{align}\label{ddjjvvjdrggddgtttdddt}
&\mathcal {E}_2(u(t))\leq C\mathcal {E}_2(u(t_1))+C\int_t^{t_1}\|q(\tau,\cdot)\|_{L^1}{\rm d}\tau\nonumber\\
&\leq C\mathcal {E}_2(u(t_1))+C\int_t^{t_1}\langle \tau\rangle^{-3/2}E_4^{1/2}(u(\tau))\mathcal {E}_2(u(\tau)){\rm d}\tau,~~0\leq t\leq t_1.
\end{align}
\subsection{Conclusion of the proof}
Recall that in Section \ref{SEC3333}, it has been shown that $u$, the global solution to the Cauchy problem \eqref{Cauchy}--\eqref{Cajujiniu}, will scatter. That is,  we have
\begin{equation}\label{iioo33}
\|\partial u(t)-\partial \overline{u}(t)\|_{H^1}\longrightarrow 0,~~~\text{as}~~t\longrightarrow \infty,
\end{equation}
where $\overline{u}$ satisfies
\begin{equation}
L\overline{u}=0~~~{\text{on}}~~\mathbb{R}^{+}\times \mathbb{R}^3
\end{equation}
for some initial data
\begin{equation}
(\overline{u}, \overline{u}_t)|_{t=0}=(\overline{u}_0, \overline{u}_1)\in \mathcal {H}^{2}.
\end{equation}

Now assume that $\overline{u}_0(x)=0, ~\overline{u}_1(x)=0,~x\in \mathbb{R}^3$, i.e.,
\begin{equation}\label{xuio9dd0}
\overline{u}(t,x)=0,~~t\geq 0,~x\in \mathbb{R}^3.
\end{equation}
Then it follows from \eqref{iioo33} and \eqref{xuio9dd0} that
for any $\overline{\varepsilon}>0$, there exists $t_1=t_1(\overline{\varepsilon})>0$, such that
\begin{align}\label{xuyao999}
\mathcal {E}^{1/2}_2(u(t_1))\leq \overline{\varepsilon}.
\end{align}
It follows from \eqref{ddjjvvjdrggddgtttdddt}, \eqref{xuyao999} and \eqref{Lowenergy} that
\begin{align}\label{ddjfffttdddt}
&\mathcal {E}_2(u(t))\leq C\overline{\varepsilon}^2+C\varepsilon\int_t^{t_1}\langle \tau\rangle^{-3/2}\mathcal {E}_2(u(\tau)){\rm d}\tau,~~0\leq t\leq t_1,
\end{align}
which implies
\begin{align}\label{ddjfffthhtdddt}
&\mathcal {E}_2(u(t))\leq C\overline{\varepsilon}^2+C\varepsilon\int_0^{t_1}\langle \tau\rangle^{-3/2}\mathcal {E}_2(u(\tau)){\rm d}\tau,~~0\leq t\leq t_1.
\end{align}
Thus by \eqref{ddjfffthhtdddt} we obtain
\begin{align}\label{ddjfffthhhhhtdddt}
\sup_{0\leq \tau\leq t_1}\mathcal {E}_2(u(\tau))\leq C\overline{\varepsilon}^2+C\varepsilon\int_0^{t_1}\langle t\rangle^{-3/2}\sup_{0\leq \tau\leq t}\mathcal {E}_2(u(\tau)){\rm d}t.
\end{align}
In view of \eqref{ddjfffthhhhhtdddt}, applying of the Gronwall inequality results in
\begin{align}\label{ddjfffthhhjjhhtdddt}
\sup_{0\leq \tau\leq t_1}\mathcal {E}_2(u(\tau))\leq C\overline{\varepsilon}^2e^{C\varepsilon}.
\end{align}
By the arbitrariness of $\overline{\varepsilon}$, we have
\begin{align}
\sup_{0\leq \tau\leq t_1}\mathcal {E}_2(u(\tau))=0.
\end{align}
Particularly,
\begin{align}
\mathcal {E}_2(u(0))=0,
\end{align}
which gives
\begin{equation}
u(0,x)=u_0(x)=0,~~\partial_tu(0,x)=u_1(x)=0,~~x\in \mathbb{R}^3.
\end{equation}
By the uniqueness of global classical solution to the Cauchy problem \eqref{Cauchy}--\eqref{Cajujiniu}, we get
\begin{equation}
u(t,x)=0,~~t\geq 0,~x\in \mathbb{R}^3.
\end{equation}
\begin{rem}
From the above discussion, we see that in order to close the argument in the rigidity part, it is necessary to concern second order energy (see \eqref{ddjfffthhhhhtdddt}). This is also the reason why we use second order derivatives in the definitions of asymptotically free and scattering properties.
\end{rem}

\end{document}